\pdfoutput=1

\documentclass[twocolumn]{autart}    
\usepackage{amsfonts}
\usepackage{indentfirst}
\usepackage{amssymb}
\usepackage{graphicx}          
\usepackage[authoryear,longnamesfirst]{natbib}
\usepackage{changes}
\usepackage{changes}
\usepackage{color}
\usepackage{epsfig}
\usepackage{mathrsfs}
\usepackage{epstopdf}
\usepackage{amsmath}
\usepackage[noend]{algpseudocode}
\usepackage{algorithmicx,algorithm}
\usepackage{enumitem}
\usepackage[none]{hyphenat}
\DeclareMathOperator*{\esssup}{ess\,sup}

\newtheorem{proposition}{Proposition}
\newtheorem{assumption}{Assumption}
\newtheorem{remark}{Remark}
\newtheorem{definition}{Definition}
\newtheorem{lemma}{Lemma}
\newtheorem{theorem}{Theorem}
\newtheorem{corollary}{Corollary}
\newtheorem{example}{Example}
\newtheorem{conjecture}{Conjecture}
\newenvironment{proof}{\emph{Proof:}}

\begin{document}
\setlength{\parindent}{2em}
\begin{frontmatter}

\title{Stability and Bounded Real Lemmas of Discrete-Time MJLSs with the Markov Chain on a Borel Space} 

\thanks[footnoteinfo]{Corresponding author: Ting Hou. Tel. +86-13455248030.}

\author[sdnu]{Chunjie Xiao}\ead{$xiaocj6\_sd@163.com$},
\author[sdnu]{Ting Hou}\ead{$ht\_math@sina.com$},
\author[sdust]{Weihai Zhang}\ead{$w\_hzhang@163.com$},

\address[sdnu]{College of Mathematics and Statistics, Shandong Normal University,
Jinan 250014, Shandong Province, China}
\address[sdust]{College of Electrical Engineering and Automation, Shandong University of Science and Technology,
Qingdao 266590, Shandong Province, China}


\begin{abstract}             
In this paper, exponential stability of discrete-time Markov jump linear systems (MJLSs)
with the Markov chain on a Borel space $(\Theta, \mathcal{B}(\Theta))$
is studied, and bounded real lemmas (BRLs) are given.
The work generalizes the results from the previous literature that considered only the Markov chain taking values in a countable set
to the scenario of an uncountable set and provides unified approaches for describing exponential stability and $H_{\infty}$ performance of MJLSs.
This paper covers two kinds of exponential stabilities: one is exponential mean-square stability with conditioning (EMSSy-C),
and the other is exponential mean-square stability (EMSSy).
First, based on the infinite-dimensional operator theory,
the equivalent conditions for determining these two kinds of stabilities are shown respectively by the exponentially stable evolutions
generated by the corresponding bounded linear operators on different Banach spaces,
which turn out to present the spectral criteria of EMSSy-C and EMSSy.
Furthermore, the relationship between these two kinds of stabilities is discussed. Moreover,
some easier-to-check criteria are established for EMSSy-C of MJLSs in terms of the existence of
uniformly positive definite solutions of Lyapunov-type equations or inequalities.
In addition, BRLs are given separately in terms of the existence of solutions of
the $\Theta$-coupled difference Riccati equation for the finite horizon case
and algebraic Riccati equation for the infinite horizon case,
which facilitates the $H_{\infty}$ analysis of MJLSs with the Markov chain on a Borel space.
\sloppy{}
\end{abstract}
\end{frontmatter}

\section{Introduction}

Switching systems are usually considered to model real-world systems that undergo sudden changes in structures
due to environmental disturbances and component failures or repairs.
As a special kind of stochastic switching systems governed by Markov processes,
Markov jump linear systems (MJLSs) have gained much attention,
and some theoretical results on stability, observability, detectability, state estimation, robust control, and optimal control
have been developed for this kind of systems.
We do not intend to be exhaustive here but only mention
\cite{BookDragan2010, SCIchinadeng2023, Hout2016auto, Collin2016, Meynsp2009, ShiT2022, SiB2020, Zhoubin2018}
as well as the references therein as several samples of the theoretical studies on issues related to this paper.
Moreover, applications of the above theoretical results can be found in a wide range of fields,
including communications (see, \cite{Liuming2009, XueM2021}), robotics (see, \cite{Andres2021}),
and new energy generation (see, \cite{Lin2016, Sworder1983}).

In real-world systems, the presence of exogenous disturbance often leads to oscillation or even results in worsen control quality.
To investigate the disturbance attenuation for disturbed systems, in recent decades,
$H_{\infty}$ analysis has been the topic of extensive investigation.
As a key result in $H_{\infty}$ analysis,
the infinite horizon bounded real lemma (BRL) in time-domain analysis estimates the $H_{\infty}$ norm
of perturbation operator for the internally stable systems
relying on the existence of stabilizing solutions of the algebraic Riccati equations (AREs) or the viability of some linear matrix inequalities.
Further, it is important that BRLs have been widely applied in filter and controller designs.
For instance, BRLs were presented in \cite{Seiler2003} for homogeneous and in \cite{Aberkane2013} for nonhomogeneous MJLSs
where the Markov chain takes values in a finite set (finite MJLSs).
Regarding the applications in state estimation, the $H_{\infty}$ filter of MJLSs with the piecewise homogeneous Markov chain
was designed in \cite{Zhanglx2009} by presenting the BRL for the corresponding filtering error systems.
As for $H_{\infty}$ control problems, after providing a BRL for $H_{\infty}$ performance analysis of MJLSs with partial observations,
a detector-based $H_{\infty}$ controller design method was introduced in \cite{Marcos2018auto}.

The literature mentioned above focused only on analysis and synthesis problems of finite MJLSs,
and in effect, most of the existing literature on MJLSs did as well.
However, one should note that different Markov chain state spaces might result in different research findings.
For example, the notions of stochastic stability and mean-square stability are equivalent for finite MJLSs (see, \cite{Fengxial1992}),
but no longer equivalent when the Markov chain state space is countably infinite.
One may refer to \cite{Fragosojack2002} and \cite{Costa1995} for the cases of continuous-time and discrete-time systems, respectively.
For that reason, the control problems of MJLSs with countably infinite Markov jump parameters have been studied intensively
(see, \cite{BookDragan2014, SCIchinaHou2021, Hou2016AC, Marcosgtmf2011, Vitoricam2014}).
Regardless of whether the state space of a Markov chain is a finite set or a countably infinite set,
it just takes values in a discrete set, and none of the above results can be directly applied to MJLSs with the continuous state Markov chain
due to the difference between the sum of finite or countably infinite terms and the integral.
However, in real-world applications, a continuous state Markov chain might be more suitable for describing some general phenomena,
such as the continuous-valued random delays that exist in networked control systems
randomly determined by the last few delays (see, \cite{Masashi2018}).
Naturally, when the Markov chain is on a Borel space,
which is a general (not necessarily countable, e.g., uncountable) state space, including the case of continuous state space,
some general findings on analysis and synthesis of MJLSs have been presented over the last decade.
For example, \cite{Ioannis2014, Li2012, Costa2014} focused on mean-square stability,
uniformly exponential almost sure stability, and stochastic stability of MJLSs, respectively.
In addition, some meaningful problems of MJLSs,
such as LQ control problems (see, \cite{Ioannis2014, Costa2015, Costa2016}) and filtering problems (see, \cite{Costa2016, Costa2017}),
were studied.

In this paper, we pay attention to exponential stability and disturbance attenuation property of discrete-time MJLSs
with the Markov chain on a Borel space $(\Theta, \mathcal{B}(\Theta))$.
The main contributions are given as follows:
\begin{itemize}
\item This paper extends the previous results of MJLSs with the Markov chain on a countable set to an uncountable set.
      The extension is achieved based on the analysis of the Borel-measurable matrix-valued functions with the help of measure theory,
      such as ensuring the integrability of solutions of $\Theta$-coupled Lyapunov-type equations and AREs,
      which makes an essential difference from the previous literature.
      In fact, the Markov chain state space could be finite or countably infinite,
      only that these two special cases are no longer the key focus of our study.
      For one thing, as mentioned earlier, these two cases have been intensively studied.
      For another thing, under the analytical framework established in this paper,
      some results obtained previously in the literature for these two cases are readily apparent.
\item Based on the infinite-dimensional operator theory, this paper establishes the equivalence
      between the exponential mean-square stability with conditioning (EMSSy-C) of MJLSs and the  exponentially stable anticausal evolution (ESAE)
      generated by the bounded linear operator $\mathcal{T}_{A}$ on Banach space $\mathbb{SH}_{\infty}^{n}$.
      In fact, it is reasonable to propose the concept of EMSSy-C because in some real-world systems
      which subsystem the switching system belongs to at the initial moment is directly available or even achievable by control.
      Moreover, a spectral criterion and some easier-to-check Lyapunov-type criteria are established for testing EMSSy-C of MJLSs. Regarding exponential mean-square stability (EMSSy), a criterion is given by the exponentially stable causal evolution (ESCE) generated by
      operator $\mathcal{L}_{A}$ on Banach space $\mathbb{SH}_{1}^{n}$.
      As a result, a spectral criterion for EMSSy of MJLSs is directly obtained.
\item The BRLs are characterized by the $\Theta$-coupled difference Riccati equation (DRE) for the finite horizon case
      and ARE for the infinite horizon case (regarding the $\Theta$-coupled ARE, one can refer to \cite{Costa2015}),
      which develops a general method for the $H_{\infty}$ analysis of MJLSs with the Markov chain
      on a Borel space $(\Theta, \mathcal{B}(\Theta))$ in a unified way.
      For the infinite horizon case,
      it is proved that internal stability of the system guarantees external stability,
      and therefore the associated perturbation operators can be given.
      To measure the effect of unknown perturbations on the output in the worst-case scenario for MJLSs,
      the $H_{\infty}$ norm of the given perturbation operator is no longer provided by a $\sup$ operator
      but is given by an $\esssup$ operator which appears to be new.
      Furthermore, to the best of the authors' knowledge,
      this is the first time that a sufficient and necessary condition is presented for a given MJLS
      with the Markov chain on a general Borel space
      to ensure that the $H_{\infty}$ norm of the associated perturbation operator is below a prescribed level $\gamma>0$.
      Moreover, we set up an iterative algorithm for solving the associated $\Theta$-coupled ARE.
      Besides, it is noteworthy here that $H_{\infty}$ analysis is the basis of $H_{\infty}$ synthesis.
      The BRLs obtained can be subsequently applied to solve the system synthesis problems
      such as $H_{\infty}$ filter and $H_{\infty}$ controller designs of MJLSs.
\end{itemize}

The remainder is arranged as follows:
In Section \ref{Preliminaries and Auxiliary Results},
the dynamic model of MJLSs with the Markov chain on a general Borel space $(\Theta, \mathcal{B}(\Theta))$
and some auxiliary results on the bounded linear operators are described;
Section \ref{Exponential Stability} deals with exponential stability;
In Section \ref{Bounded Real Lemma}, we further explore $H_{\infty}$ performance of the systems,
and BRLs in the finite and infinite horizon are derived;
Examples are given in Section \ref{Example};
Section \ref{Conclusion} finishes the paper with some concluding remarks.

\section{Preliminaries and Auxiliary Results}\label{Preliminaries and Auxiliary Results}

\textbf{Notations}.
$\mathbb{N}^{+}:=\{1,2,\cdots\}$;
$\mathbb{N}:=\{0\}\bigcup\mathbb{N}^{+}$;
for integers $i$ and $j$, $\overline{i,j}:=\{i,i+1,i+2,\cdots,j\}$;
$\mathbb{R}^{n}$ is the $n$-dimensional real Euclidean space with the Euclidean norm $\|\cdot\|$.
For Banach spaces $\mathbb{M}$ and $\bar{\mathbb{M}}$, $\mathbb{B}(\mathbb{M},\bar{\mathbb{M}})$ represents the Banach space of
all bounded linear operators of $\mathbb{M}$ into $\bar{\mathbb{M}}$ with the uniform induced norm denoted by $\|\cdot\|$,
i.e., for $\mathcal{L}\in\mathbb{B}(\mathbb{M},\bar{\mathbb{M}})$, $\|\mathcal{L}\|:=\sup_{x\in\mathbb{M},\ x\neq0}\{{\|\mathcal{L}x\|}/{\|x\|}$\};
for simplicity, $\mathbb{B}\left(\mathbb{M}\right):=\mathbb{B}(\mathbb{M},\mathbb{M})$.
Let $\sigma(\mathcal{L})$ be the spectral set of $\mathcal{L} \in \mathbb{B}(\mathbb{M})$,
and the spectral radius of $\mathcal{L}$ is defined to be $r_{\sigma}(\mathcal{L}):=\sup_{\lambda\in{\sigma(\mathcal{L})}}|\lambda|$. For Banach space $\mathbb{M}$, a convex cone $\mathbb{M}^{+}\subset \mathbb{M}$ induces an ordering ``$\leq$'' on $\mathbb{M}$
by $V\leq \bar{V}$ if and only if (iff) $\bar{V}-V\in \mathbb{M}^{+}$.
An operator $\mathcal{L} \in \mathbb{B}(\mathbb{M})$ is called a positive operator,
denoted by $\mathcal{L}\geq0$, if $\mathcal{L}\mathbb{M}^{+}\subset\mathbb{M}^{+}$.
In particular, $\mathbb{R}^{m\times n}:=\mathbb{B}\left(\mathbb{R}^{n}, \mathbb{R}^{m}\right)$ denotes the $m \times n$ real matrix space.
As usual, the $n\times n$ identity matrix is expressed by $I_{n\times n}$.
For $Q\in\mathbb{R}^{m\times n}$, $Q'$ means the transpose of $Q$;
$\|Q\|:=[\lambda_{\max}(Q'Q)]^{\frac{1}{2}}$; $\mathbb{S}^{n}:=\{Q\in{\mathbb{R}^{n\times n}}|Q'=Q\}$.
For $Q_{1}\in{\mathbb{S}^{n}},\ Q_{2}\in{\mathbb{S}^{n}}$,
$Q_{1}<Q_{2}\ (Q_{1}\leq Q_{2})$ indicates that $Q_{1}-Q_{2}$ is a negative definite (negative semidefinite) matrix.
$\mathbb{S}^{n+*}:=\{Q\in{\mathbb{S}^{n}}|Q\geq \alpha I_{n\times n}\  \text{for}\  \text{some}\ \alpha>0\}$.
In what follows, we let $n$ be the associated dimension of the matrix, and $n$ is a constant.

The following result will be mentioned repeatedly:
\begin{proposition}\label{235QtrqnormQ}\citep{Costa2014}
Given any $Q\in{\mathbb{S}^{n}}$ with $Q\geq0$, we have $\|Q\|\leq tr(Q) \leq n \|Q\|$, where $tr(\cdot)$ denotes the trace operator.
\end{proposition}

In this paper, $\Theta$ is assumed to be a Borel subset of a Polish space (i.e., a separable and complete metric space).
The Borel space $(\Theta, \mathcal{B}(\Theta))$ is defined as $\Theta$,
together with its Borel $\sigma$-algebra $\mathcal{B}(\Theta)$.
$\mu$ is a $\sigma$-finite measure on $\mathcal{B}(\Theta)$.
By Theorem 13.1 in \cite{Kechris1995}, the topology of the Polish space can be extended to
a new topology with the same Borel sets in which $\Theta$ is clopen (i.e., open and closed), and so is Polish (in the relative topology).
Further, by Theorem 8.3.6 in \cite{Cohn2013}, if $\Theta$ is uncountable (which is the scenario of our interest),
then it is Borel isomorphic to $\mathbb{R}^{1}$.

On a probability space $(\Omega, \mathfrak{F}, \mathcal{P})$,
define a Markov chain $\{\vartheta(k)\}_{k\in\mathbb{N}}$ taking values in $\Theta$
with the initial distribution given by a probability measure $\hat{\mu}$
and the stochastic kernel $\mathcal{G}(\cdot|\cdot)$ satisfying
$\mathcal{G}(\Delta|\vartheta(k)):=\mathcal{P}(\vartheta(k+1)\in \Delta|\vartheta(k))$
almost surely $(a.s.),$  $ \forall k\in\mathbb{N}$ and $\Delta\in\mathcal{B}(\Theta)$.
Suppose that for any $\ell\in\Theta$,
$\mathcal{G}(\cdot|\ell)$ has a density $g(\cdot|\ell)$ with respect to $\mu$,
that is, $\mathcal{G}(\Delta|\ell)=\int_{\Delta}g(s|\ell)\mu(ds)$ for any $\Delta\in\mathcal{B}(\Theta).$
We now explain the reasonableness of the assumption on the Markov chain when $\Theta$ is Borel.
Indeed, by Theorem 6.3 in \cite{Kallenberg2002}, the stochastic kernel $\mathcal{G}(\cdot|\cdot)$ exists.
Moreover, the Kolmogorov consistency theorem (for example, see Theorem 10.6.2 in \cite{Cohn2013})
allows us to construct a $\Theta$-valued Markov chain
with the initial distribution $\hat{\mu}$ and the stochastic kernel $\mathcal{G}(\cdot|\cdot)$.

We say that $P=\{P(\ell)\}_{\ell\in\Theta}\in\mathbb{H}^{m\times n}$ ($P\in\mathbb{SH}^{n}$)
if  $P(\cdot): \Theta\rightarrow \mathbb{R}^{m\times n}$ ($P(\cdot): \Theta\rightarrow \mathbb{S}^{n}$) is measurable.
For $P$, $R\in\mathbb{SH}^{n}$, $P\leq R$ means that $P(\ell)\leq R(\ell)$ holds for $\mu$-almost all $\ell\in\Theta$.
$\mathbb{H}^{m\times n}_{1}=\{P\in\mathbb{H}^{m\times n}\big| \|P\|_{1}:=\int_{\Theta}\|P(\ell)\|\mu(d\ell)<\infty\}$;
$\mathbb{H}^{m\times n}_{\infty}\!=\!\{P\in\mathbb{H}^{m\times n}\big| \|P\|_{\infty}\!:=\!\esssup\{\|P(\ell)\|; \ell\in{\Theta}\}<\infty\}$.
For convenience, $\mathbb{H}^{n}_{1}:=\mathbb{H}^{n\times n}_{1}$;
$\mathbb{H}^{n}_{\infty}:=\mathbb{H}^{n\times n}_{\infty}$.
$\mathbb{SH}^{ n}_{1}=\{P\in\mathbb{SH}^{n}\big| \|P\|_{1}<\infty\}$;
$\mathbb{SH}^{n}_{\infty}=\{P\in\mathbb{SH}^{n}\big| \|P\|_{\infty}<\infty\}$.
In fact, $(\mathbb{H}_{1}^{m\times n},\|\cdot\|_{1})$ and $(\mathbb{H}^{m\times n}_{\infty},\|\cdot\|_{\infty})$
are Banach spaces (see \cite{Costa2014}).
On Banach space $(\mathbb{SH}_{1}^{n},\|\cdot\|_{1})$,
we consider the order relation induced by the convex cone $\mathbb{H}_{1}^{n+}\subset\mathbb{SH}_{1}^{n}$,
where $\mathbb{H}_{1}^{n+}=\{P\in \mathbb{SH}_{1}^{n}|P(\ell)\geq 0$ holds for $\mu$-almost all $\ell\in\Theta\}$.
On Banach space $(\mathbb{SH}_{\infty}^{n},\|\cdot\|_{\infty})$,
we consider the order relation induced by the convex cone $\mathbb{H}_{\infty}^{n+}\subset\mathbb{SH}_{\infty}^{n}$,
where $\mathbb{H}_{\infty}^{n+}=\{P\in \mathbb{SH}_{\infty}^{n} |P(\ell)\geq0$ holds for $\mu$-almost all $\ell\in\Theta\}$.
It is obvious that $(\mathbb{SH}_{1}^{n},\|\cdot\|_{1})$ and $(\mathbb{SH}_{\infty}^{n},\|\cdot\|_{\infty})$
are infinite-dimensional real ordered Banach spaces.
$\mathbb{H}_{\infty}^{n+*}=\{P\in \mathbb{H}_{\infty}^{n+} |P(\ell) \geq \alpha I_{n\times n}$
holds for $\mu$-almost all $\ell\in \Theta$, for some $\alpha>0\}$.
Let $\mathcal{I}=\{\mathcal{I}(\ell)\}_{\ell\in\Theta}$, where {$\mathcal{I}(\ell)=I_{n\times n}$.
$\mathcal{I}\in\mathbb{H}_{\infty}^{n+}$ and  is the identity element of $\mathbb{SH}_{\infty}^{n}$.
$\mathcal{I}_{\mathbb{SH}_{\infty}^{n}}$ denotes the identity operator on $\mathbb{SH}_{\infty}^{n}.$
Throughout this paper, when analyzing the functions on the given measurable spaces,
all the properties (including all the equations and inequalities) are
in the sense of $\mu$-almost everywhere on $\Theta$ ($\mu\text{-}a.e.$),
or for $\mu$-almost all $\ell\in\Theta$, if there is no special reminder.
\begin{proposition}\label{decomposition}
For any $P\in\mathbb{SH}_{1}^{n}$ $(P\in\mathbb{SH}_{\infty}^{n})$,
there exist $P^{+},\ P^{-}\in\mathbb{H}^{n+}_{1}$ $(P^{+},\ P^{-}\in\mathbb{H}^{n+}_{\infty})$ such that $P=P^{+}-P^{-}$
and $\|P\|_{1}\geq \max\{\|P^{+}\|_{1},\ \|P^{-}\|_{1}\}$ $(\|P\|_{\infty}\geq \max\{\|P^{+}\|_{\infty},\ \|P^{-}\|_{\infty}\})$.
\end{proposition}
\begin{proof}
For any $P\in\mathbb{SH}_{1}^{n}$ ($P\in\mathbb{SH}_{\infty}^{n}$), define
$\hat{P}=\{\hat{P}(\ell)\}_{\ell\in\Theta}$ with $\hat{P}(\ell)=\|P(\ell)\| I_{n\times n}$.
Clearly, $\hat{P}(\cdot)$ is measurable.
Let
$P^{+}(\ell)=[\hat{P}(\ell)+P(\ell)]/2$ and
$P^{-}(\ell)=[\hat{P}(\ell)-P(\ell)]/2,\ \forall \ell\in\Theta.$
The reader can verify that $P=P^{+}+P^{-}$ may serve as a valid decomposition.
\end{proof}

Consider the discrete-time MJLS
\begin{equation*} \label{system}
\Phi_{v}:\left\{
\begin{array}{ll}
x(k+1)=A(\vartheta(k))x(k)+B(\vartheta(k))v(k), \\
y(k)=C(\vartheta(k))x(k)+D(\vartheta(k))v(k),\ \forall k\in\mathbb{N},\\
\ x(0)=x_{0}\in\mathbb{R}^{n},\  \vartheta(0)=\vartheta_{0},
\end{array}
\right.
\end{equation*}
where $x(k)\in{\mathbb{R}^{n}}$ is the system state;
$v(k)\in{\mathbb{R}^{r}}$ and $y(k)\in{\mathbb{R}^{m}}$ are the input and output;
$x_{0}$ is a deterministic vector in $\mathbb{R}^{n}$.
$\mathfrak{F}_{k}$ denotes the $\sigma$-field generated by $\{x(0),\ \vartheta(0)$, $x(1)$, $\vartheta(1),$ $\cdots,$ $x(k),$ $\vartheta(k)\}$.
The following assumptions are given with regard to $\Phi_{v}$:
\begin{assumption}\label{286ass}
(i)~$A\in\mathbb{H}^{n}_{\infty}$, $B\in\mathbb{H}^{n\times r}_{\infty}$,
$C\in\mathbb{H}^{m\times n}_{\infty}$, and $D\in\mathbb{H}^{m\times r}_{\infty}$;\\
(ii)~$\vartheta(k)$ is directly accessible at every $k\in\mathbb{N}$;\\
(iii)~$C(\ell)'D(\ell)=0$ holds for $\mu$-almost all $\ell\in\Theta$;\\
(iv)~The initial distribution $\hat{\mu}$ is absolutely continuous with respect to $\sigma$-finite measure $\mu$.
\end{assumption}

Considering $(iv)$ of Assumption \ref{286ass},
and applying the Radon-Nikodym theorem (for example, see Theorem 32.2 in \cite{BookBillingsley1995}),
it follows that there is $\nu(\ell)\geq0$ satisfying $\hat{\mu}(M)=\int_{M}\nu(\ell)\mu(d\ell)$ for any $M\in{\mathcal{B}(\Theta)}$,
i.e., $\nu$ plays the role of the Radon-Nikodym derivative of $\hat{\mu}$ with respect to $\mu$.
Therefore, we can properly make the assumptions as follows:
\begin{assumption}\label{Ass2902}
~$\nu(\ell)>0$ $\mu\text{-}a.e.$.
\end{assumption}
\begin{assumption}\label{1977assgtl0}
$\int_{\Theta}g(\ell|s)\mu(ds)>0$ $\mu\text{-}a.e.$.
\end{assumption}

In the remainder, we will show several auxiliary results for some given bounded linear operators.
For $K\in\mathbb{H}^{n}_{\infty}$, $V\in\mathbb{H}^{n}_{1},\ U\in\mathbb{H}^{n}_{\infty}$, and $\ell\in\Theta$, define
\begin{eqnarray}
&&\ \ \ \ \ \ \mathcal{L}_{K}(V)(\ell):=\int_{\Theta}g(\ell|s)K(s)V(s)K(s)'\mu(ds),\label{LK}\\
&&\ \ \ \ \ \ \mathcal{E}(U)(\ell):=\int_{\Theta}g(s|\ell)U(s)\mu(ds),\nonumber\\
&&\ \ \ \ \ \ \mathcal{T}_{K}(U)(\ell):=K(\ell)'\mathcal{E}(U)(\ell)K(\ell),\label{TK}
\end{eqnarray}
and the bounded bilinear operator by $\langle V; U\rangle:=\int_{\Theta}tr(V(s)'U(s))\mu (ds).$
The following proposition shows some essential properties about the above operators:
\begin{proposition}\label{operatorsspace}\citep{Costa2014}
For any $K\in\mathbb{H}^{n}_{\infty}$, $V\in\mathbb{H}_{1}^{n}$, and $U\in{\mathbb{H}_{\infty}^{n}}$, \\
(i)~$\mathcal{L}_{K}\in\mathbb{B}(\mathbb{H}_{1}^{n})$ and $r_{\sigma}(\mathcal{L}_{K})\leq\|\mathcal{L}_{K}\|\leq\|K\|^{2}_{\infty}$;\\
(ii)~$\mathcal{T}_{K}\in\mathbb{B}(\mathbb{H}_{\infty}^{n})$ and $r_{\sigma}(\mathcal{T}_{K})\leq\|\mathcal{T}_{K}\|\leq\|K\|^{2}_{\infty}$;\\
(iii)~$\langle \mathcal{L}_{K}(V); U\rangle= \langle V; \mathcal{T}_{K}(U)\rangle$.
\end{proposition}
\begin{remark}\label{SH1SHinf441}
Specially, if we let $K\in\mathbb{H}^{n}_{\infty}$, $V\in\mathbb{SH}^{n}_{1},$ and $U\in\mathbb{SH}^{n}_{\infty}$,
then $\mathcal{L}_{K}$ and $\mathcal{T}_{K}$ satisfy $\mathcal{L}_{K}(V)\in\mathbb{SH}_{1}^{n}$
and $\mathcal{T}_{K}(U)\in\mathbb{SH}_{\infty}^{n}$, respectively.
Indeed, $\mathbb{SH}_{1}^{n}$ is an invariant linear subspace of Banach space $\mathbb{H}_{1}^{n}$ under the operator $\mathcal{L}_{K}$,
while $\mathbb{SH}^{n}_{\infty}$ is an invariant linear subspace of Banach space $\mathbb{H}^{n}_{\infty}$ under the operator $\mathcal{T}_{K}$.
Consequently, we can draw the same conclusions on $\mathbb{SH}^{n}_{1}$ and $\mathbb{SH}^{n}_{\infty}$ as Proposition \ref{operatorsspace}.
Moreover, $\mathcal{L}_{K}$ and $\mathcal{T}_{K}$ are positive operators
on ordered Banach spaces $\mathbb{SH}^{n}_{1}$ and $\mathbb{SH}_{\infty}^{n}$, respectively.
\end{remark}

In what follows, if not specified, $\mathcal{L}_{A}$ and $\mathcal{T}_{A}$ will be restricted on
$\mathbb{B}(\mathbb{SH}_{1}^{n})$ and $\mathbb{B}(\mathbb{SH}_{\infty}^{n})$, respectively.
Here, $\mathcal{L}_{A}$ and $\mathcal{T}_{A}$ are defined by \eqref{LK} and \eqref{TK} with $K=A$.

For the sake of concentrating on internal stability of $\Phi_{v}$, let $(\mathbf{A};\mathcal{G})$ represent the following system:
\begin{equation*}\label{ag}
\bar{x}(k+1)\!=\!A(\vartheta(k))\bar{x}(k), \ \forall k\in\mathbb{N},\ \bar{x}(0)=x_{0},\  \vartheta(0)=\vartheta_{0}.
\end{equation*}
Define $\Gamma(k_{2},k_{1})=A(\vartheta(k_{2}-1))A(\vartheta(k_{2}-2))\cdots A(\vartheta(k_{1}))$ for $k_{1}<k_{2}$,
and $\Gamma(k_{1},k_{1})=I_{n\times n}$, as the fundamental matrix solution of $(\mathbf{A};\mathcal{G})$, $k_{1},k_{2}\in \mathbb{N}$.
Thus, $\forall k\in\mathbb{N}$, $\bar{x}(k)=\Gamma(k,0)x_{0}$.
Define $\bar{X}(k)\!=\!\{\bar{X}(k,\ell)\}_{\ell\in{\Theta}}$, $k\in\mathbb{N}$ by
\begin{equation}\label{BArxX}
\bar{X}(k,\ell)=\mathbb{E}\{\bar{x}(k)\bar{x}(k)'
g(\ell|\vartheta(k-1))\},\ \forall k\in \mathbb{N}^{+},
\end{equation}
and $\bar{X}(0,\ell)=X_{0}(\ell),\ \forall\ell\in\Theta$, where $X_{0}(\ell)=x_{0}x_{0}'\nu(\ell)$.
Clearly, $\forall k\in\mathbb{N}$, $\bar{X}(k)\in{\mathbb{H}_{1}^{n+}}$.
Using $\mathcal{L}_{A}$, the recursive form of $\bar{X}(\cdot)$ is established in the next proposition.
For more details on $(i)$,  please refer to \cite{Costa2014}, and $(ii)$ follows directly from $(i)$.
\begin{proposition}\label{420propo}
(i)~For every $k\in \mathbb{N},$ $\bar{X}(k)\in{\mathbb{H}_{1}^{n+}}$ given by \eqref{BArxX} satisfies the linear difference equation:
$\bar{X}(k+1)=\mathcal{L}_{A}(\bar{X}(k))$ with $\bar{X}(0)=X_{0}=\{X_{0}(\ell)\}_{\ell\in\Theta}$;\\
(ii)~For every $k\in \mathbb{N},$ $\bar{X}(k)=\mathcal{L}_{A}^{k}(X_{0})$.
\end{proposition}

The following definitions are from \cite{BookDragan2010}.
For all integers $ k_{1}\leq k_{2}$, the causal evolution operator $S_{A}(k_{2},k_{1})$ generated by $\mathcal{L}_{A}$ is given by $S_{A}(k_{2},k_{1})=\mathcal{L}_{A}^{k_{2}-k_{1}}$ for $k_{1}<k_{2}$
and $S_{A}(k_{1},k_{1})=\mathcal{I}_{\mathbb{SH}_{1}^{n}}$,
where $\mathcal{I}_{\mathbb{SH}_{1}^{n}}$ is the identity operator on $\mathbb{SH}_{1}^{n}.$
The anticausal evolution operator $T_{A}(k_{1},k_{2})$ generated by $\mathcal{T}_{A}$
is given by $T_{A}(k_{1},k_{2})=\mathcal{T}_{A}^{k_{2}-k_{1}}$ for $k_{1}<k_{2}$ and $T_{A}(k_{2},k_{2})=\mathcal{I}_{\mathbb{SH}_{\infty}^{n}}$.
\begin{definition}
(i)~We call that the causal evolution operator $S_{A}(k_{2},k_{1})$ generated by $\mathcal{L}_{A}$ is exponentially stable (or $\mathcal{L}_{A}$ generates an ESCE)
if there exist $\beta \geq 1$, $\alpha \in(0,1)$ such that
$\|S_{A}(k_{2},k_{1})\| \leq \beta \alpha^{k_{2}-k_{1}}$
for all integers $ k_{1}\leq k_{2}$;\\
(ii)~We call that the anticausal evolution operator $T_{A}(k_{1},k_{2})$ generated by $\mathcal{T}_{A}$ is exponentially stable (or  $\mathcal{T}_{A}$ generates an ESAE)
if there exist $\beta \geq 1$, $\alpha \in(0,1)$ such that
$\|T_{A}(k_{1},k_{2})\| \leq \beta \alpha^{k_{2}-k_{1}}$
for all integers $k_{1}\leq k_{2}$.
\end{definition}

\section{Exponential Stability}\label{Exponential Stability}

In this section, we will deal with
 exponential stability of $(\mathbf{A};\mathcal{G})$.
\begin{definition}\label{def596emssc}
(i)~We call that $(\mathbf{A};\mathcal{G})$ is exponentially mean-square stable with conditioning (EMSS-C),
if there exist $\beta \geq 1$, $\alpha \in(0,1)$ such that
$\mathbb{E}\{\left.\|\bar{x}(k)\|^{2}|\vartheta_{0}=\ell\right\} \leq \beta \alpha^{k}\|x_{0}\|^{2}$ $\mu\text{-}a.e.$
for every $k \in \mathbb{N}$ and $x_{0} \in \mathbb{R}^{n}$;\\
(ii)~We call that $(\mathbf{A};\mathcal{G})$ is exponentially mean-square stable (EMSS),
if there exist $\beta \geq 1$, $\alpha \in(0,1)$ such that
$\mathbb{E}\{\left.\|\bar{x}(k)\|^{2}\right\} \leq \beta \alpha^{k}\|x_{0}\|^{2}$
for every $k \in \mathbb{N}$, $x_{0}\in \mathbb{R}^{n}$, and $\vartheta_{0}$.
\end{definition}

According to Definition \ref{def596emssc}, we could unequivocally draw the following result, which shows that EMSSy-C is stronger than EMSSy.
\begin{proposition}\label{co973EMSSEMSSC}
$(\mathbf{A};\mathcal{G})$ is EMSS provided $(\mathbf{A};\mathcal{G})$ is EMSS-C.
\end{proposition}
\begin{proof}
Since $\alpha$, $\beta$ are independent of $\ell$ in Definition \ref{def596emssc},
if $(\mathbf{A};\mathcal{G})$ is EMSS-C, then there exist $\beta \geq 1$, $\alpha \in(0,1)$ such that
$\mathbb{E}\{\|\bar{x}(k)\|^{2}\}=\int_{\Theta}\mathbb{E}\{\|\bar{x}(k)\|^{2}\|\vartheta_{0}=\ell\}\nu(\ell)\mu(d\ell)\leq\beta \alpha^{k}\|x_{0}\|^{2} \int_{\Theta}\nu(\ell)\mu(d\ell)=\beta \alpha^{k}\|x_{0}\|^{2},$
which ends the proof.
\end{proof}
\begin{remark}
The current idea of the proof of Proposition \ref{co973EMSSEMSSC} is acquired along the hint of an anonymous reviewer.
\end{remark}

Below we attempt to discuss the EMSSy-C of $(\mathbf{A};\mathcal{G})$ through operator $\mathcal{T}_{A}$.
A necessary lemma is put forward first.
\begin{lemma}\label{le620diedai}
For any $V=\{V(\ell)\}_{\ell\in\Theta}\in{\mathbb{SH}_{\infty}^{n}}$ and $k\in\mathbb{N}$,
\begin{equation}\label{lemma4gvkg622}
\mathbb{E}\{\Gamma(k,0)'V(\vartheta(k))\Gamma(k,0)|\vartheta_{0}\}=\mathcal{T}_{A}^{k}(V)(\vartheta_{0})\ a.s..
\end{equation}
\end{lemma}
\begin{proof}
Let us prove the result by mathematical induction.
For any $V\in{\mathbb{SH}_{\infty}^{n}}$, if $k=0$, then \eqref{lemma4gvkg622} becomes a tautology.
Suppose that \eqref{lemma4gvkg622} holds for $k=i\in\mathbb{N}^{+}$.
According to Fubini's theorem, we get that
\begin{equation}\label{638k+1vth0}
\begin{aligned}
&\mathbb{E}\{\Gamma(i+1,0)'V(\vartheta(i+1))\Gamma(i+1,0)
|\vartheta_{0}\}\\
=&\mathbb{E}\{\mathbb{E}[\Gamma(i,0)'A(\vartheta(i))'
V(\vartheta(i+1))
A(\vartheta(i))\Gamma(i,0)|\mathfrak{F}_{i}]|\vartheta_{0}\}\\
=&\mathbb{E}\{\Gamma(i,0)'A(\vartheta(i))'
\mathbb{E}[V(\vartheta(i+1))|\mathfrak{F}_{i}]
A(\vartheta(i))\Gamma(i,0)|\vartheta_{0}\}\\
=&\mathbb{E}\!\{\!\Gamma(i,\!0)'\!A(\vartheta(i))'\!
\int_{\!\Theta\!}\!V\!(s)g(s|\vartheta(i))\mu(ds)\!
A(\vartheta(i))\Gamma(i,\!0)|\!\vartheta_{0}\!\}\\
=&\mathbb{E}\{\Gamma(i,0)'\mathcal{T}_{A}(V)
(\vartheta(i))\Gamma(i,0)|\vartheta_{0}\}\ a.s..
\end{aligned}
\end{equation}
Recalling the hypothesis that \eqref{lemma4gvkg622} holds for $i$, it follows from \eqref{638k+1vth0} that
$\mathbb{E}\{\Gamma(i+1,0)'V(\vartheta(i+1))\Gamma(i+1,0)
|\vartheta_{0}\}
=\mathbb{E}\{\Gamma(i,0)'[\mathcal{T}_{A}(V)(\vartheta(i))]
\Gamma(i,0)|\vartheta_{0}\}
=\mathcal{T}_{A}^{i}[\mathcal{T}_{A}(V)](\vartheta_{0})
=\mathcal{T}_{A}^{i+1}(V)(\vartheta_{0})\ a.s.,$
i.e., \eqref{lemma4gvkg622} is true for $k=i+1$.
Thus, \eqref{lemma4gvkg622} holds for every $k\in\mathbb{N}$.
\end{proof}
\begin{theorem}\label{EMSSCiff}
$(\mathbf{A};\mathcal{G})$ is EMSS-C iff $\mathcal{T}_{A}$ generates an ESAE.
\end{theorem}
\begin{proof}
Let $\vartheta_{0}=\ell\in\Theta$ be arbitrary but fixed.
For any $V\in{\mathbb{SH}_{\infty}^{n}}$ and $k\in\mathbb{N}$, by Lemma \ref{le620diedai}, we have that
\begin{equation}\label{TAkV}
\mathbb{E}\{\Gamma(k,0)'V(\vartheta(k))\Gamma(k,\!0)
|\vartheta_{0}\!=\!\ell\}\!=\!\mathcal{T}_{A}^{k}\!(V)(\ell)\ \mu\text{-}a.e..
\end{equation}
Letting $V=\mathcal{I}$ in \eqref{TAkV}, and pre- and post-multiplying by $x_{0}'$ and $x_{0}$, respectively,
one can infer that for every $k\in\mathbb{N}$,
$\mathbb{E}\{\|\bar{x}(k)\|^{2}|\vartheta_{0}=\ell\}
=x_{0}'\mathcal{T}_{A}^{k}(\mathcal{I})(\ell)x_{0},$ where $\bar{x}(k)$ is the system state of $(\mathbf{A};\mathcal{G})$.\\
Sufficiency. If $\mathcal{T}_{A}$ generates an ESAE,
then there are $\beta \geq 1$, $\alpha \in(0,1)$ such that
$\|\mathcal{T}_{A}^{k}(\mathcal{I})(\ell)\|\leq\|\mathcal{T}_{A}^{k}\|\leq \beta\alpha^{k},\  \forall k\in\mathbb{N}.$
Noticing that $\mathcal{T}_{A}^{k}(\mathcal{I})(\ell)\in\mathbb{S}^{n}$, one can get that
$\mathbb{E}\{\|\bar{x}(k)\|^{2}|\vartheta_{0}=\ell\}
\leq \|x_{0}\|^{2}\|\mathcal{T}_{A}^{k}(\mathcal{I})(\ell)\|$.
Hence, we conclude that $(\mathbf{A};\mathcal{G})$ is EMSS-C.\\
Necessity. If $(\mathbf{A};\mathcal{G})$ is EMSS-C, then there are $\beta_{1} \geq 1$, $\alpha \in(0,1)$ such that
$\mathbb{E}\{x_{0}'\Gamma(k,0)'\Gamma(k,0)x_{0}|\vartheta_{0}=\ell\}=\mathbb{E}\{\|\bar{x}(k)\|^{2}|\vartheta_{0}=\ell\}
\leq \beta_{1}\alpha^{k}\|x_{0}\|^{2},\  \forall k\in\mathbb{N}.$
Then
\begin{equation}\label{gamma|var0leq}
\|\mathbb{E}\{\Gamma(k,0)'\Gamma(k,0)|\vartheta_{0}=\ell\}\|
\leq \beta_{1}\alpha^{k},\ \forall k\in\mathbb{N}.
\end{equation}
Let $\hat{V}\in{\mathbb{SH}_{\infty}^{n}}$ be arbitrary and satisfy $\|\hat{V}\|_{\infty}=1$.
By Proposition \ref{decomposition},
there are $\hat{V}^{+}$, $\hat{V}^{-}\in\mathbb{H}_{\infty}^{n+}$
such that $\hat{V}=\hat{V}^{+}-\hat{V}^{-}$ and $\max\{\|\hat{V}^{+}\|_{\infty}, \|\hat{V}^{-}\|_{\infty}\}\leq1$.
Thus, for every $k\in\mathbb{N}$, it is true that
\begin{equation}\label{721vv=v-}
\begin{aligned}
&\|\mathbb{E}\{\Gamma(k,0)'\hat{V}(\vartheta(k))
\Gamma(k,0)|\vartheta_{0}=\ell\}\|\\
\leq&
\|\mathbb{E}\{\Gamma(k,0)'\hat{V}^{+}(\vartheta(k))
\Gamma(k,0)|\vartheta_{0}=\ell\}\|\\
&+\|\mathbb{E}\{\Gamma(k,0)'\hat{V}^{-}(\vartheta(k))
\Gamma(k,0)|\vartheta_{0}=\ell\}\|.
\end{aligned}
\end{equation}
Let $\bar{V}\in\mathbb{H}^{n+}_{\infty}$ be arbitrary and satisfy $\|\bar{V}\|_{\infty}\leq1$.
Next, for every $k\in\mathbb{N}$, we will prove that
\begin{equation}\label{732V+legg|v}
\begin{split}
&\|\mathbb{E}\{\Gamma(k,0)'\bar{V}(\vartheta(k))\Gamma(k,0)|\vartheta_{0}=\ell\}\|\\
\leq&
\|\mathbb{E}\{\Gamma(k,0)'\Gamma(k,0)|\vartheta_{0}=\ell\}\|.
\end{split}
\end{equation}
Since $\|\bar{V}\|_{\infty}\leq1$, then for any $z\in\mathbb{R}^{n}$,
we make that $z'\bar{V}(\vartheta(k))z\leq z' I_{n\times n}z\ a.s.$,
i.e., $\bar{V}(\vartheta(k))\leq I_{n\times n}\ a.s.$.
Hence,
$\mathbb{E}\{\Gamma(k,0)'\bar{V}(\vartheta(k))\Gamma(k,0)|\vartheta_{0}=\ell\}
\leq\mathbb{E}\{\Gamma(k,0)'\Gamma(k,0)|\vartheta_{0}\!=\!\ell\}.$
\eqref{732V+legg|v} is then derived.
For every $k\in\mathbb{N}$, \eqref{gamma|var0leq}-\eqref{732V+legg|v} result in
$\|\mathbb{E}\{\Gamma(k,0)'\hat{V}(\vartheta(k))\Gamma(k,0)|\vartheta_{0}=\ell\}\|\leq\beta\alpha^{k}, $
where $\beta=2\beta_{1}$. Combining this with \eqref{TAkV},
one can deduce that
$\|\mathcal{T}_{A}^{k}(\hat{V})(\ell)\|\leq\beta\alpha^{k}$.
Therefore, $\|\mathcal{T}_{A}^{k}(\hat{V})\|_{\infty}\leq\beta\alpha^{k}.$
Then $\|\mathcal{T}_{A}^{k}\|\leq\beta\alpha^{k}$ is reached
since $\hat{V}\in\mathbb{SH}_{\infty}^{n}$ is arbitrary and $\|\hat{V}\|_{\infty}=1$.
\end{proof}

By letting $V=\mathcal{I}$ in \eqref{TAkV}, and following a similar argument as the proof of Theorem \ref{EMSSCiff},
we can prove the following corollary:
\begin{corollary}\label{coro742tai}
$\mathcal{T}_{A}$ generates an ESAE iff there are $\beta \geq 1$, $\alpha \in(0,1)$ such that
$\|\mathcal{T}_{A}^{k}(\mathcal{I})(\ell)\|\leq \beta\alpha^{k}$ $\mu\text{-}a.e.$ for every $k \in \mathbb{N}$.
\end{corollary}

In the following theorem, the EMSSy of $(\mathbf{A};\mathcal{G})$ is characterized by operator $\mathcal{L}_{A}$.
\begin{theorem}\label{EMSSCiff2}
$(\mathbf{A};\mathcal{G})$ is EMSS iff $\mathcal{L}_{A}$ generates an ESCE.
\end{theorem}
\begin{proof}
For every $k\in{\mathbb{N}}$,
it follows from (ii) of Proposition \ref{420propo} that
\begin{equation}\label{xk2853in}
\begin{aligned}
\mathbb{E}\{\|\bar{x}(k)\|^{2}\}\!=\!\langle  \bar{X}(k); \mathcal{I}\rangle \!=\!&\int_{\!\Theta\!}\!tr\{\mathcal{L}_{A}^{k}(X_{0})(s)\}\mu (ds).
\end{aligned}
\end{equation}
Sufficiency.
By Proposition \ref{235QtrqnormQ}, for every $k\in\mathbb{N}$, we have that
$\int_{\Theta}tr\{\mathcal{L}_{A}^{k}(X_{0})(s)\}\mu (ds)\leq n\int_{\Theta}\|\mathcal{L}_{A}^{k}(X_{0})(s)\|\mu (ds)$.
So, it can be made that $\mathbb{E}\{\|\bar{x}(k)\|^{2}\}\leq n\|x_{0}\|^{2}\|\mathcal{L}_{A}^{k}\|$.
Therefore, $(\mathbf{A};\mathcal{G})$ is EMSS. \\
Necessity.
By Proposition \ref{235QtrqnormQ}, one derives that
$\|\mathcal{L}_{A}^{k}(X_{0})\|_{1}$ $\leq\int_{\Theta}tr\{\mathcal{L}_{A}^{k}(X_{0})(s)\}\mu (ds).$
From this fact, \eqref{xk2853in}, and the assumption that $(\mathbf{A};\mathcal{G})$ is EMSS,
one can get that there are $\beta \geq 1$, $\alpha \in(0,1)$ such that
\begin{equation}\label{880emss}
\|\mathcal{L}_{A}^{k}(X_{0})\|_{1}\leq\beta\alpha^{k}\|x_{0}\|^{2}\leq n\beta\alpha^{k}\|X_{0}\|_{1}
\end{equation}
holds for every $k\in\mathbb{N}$, $x_{0} \in \mathbb{R}^{n}$.
For any $P\in\mathbb{SH}_{1}^{n}$, by Proposition \ref{decomposition}, there exist $P^{+},\ P^{-}\in{\mathbb{H}_{1}^{n+}}$ such that
$P=P^{+}-P^{-}$ and $\|P\|_{1}\geq \max\{\|P^{+}\|_{1},\|P^{-}\|_{1}\}$.
Therefore,
$\|\mathcal{L}_{A}^{k}(P)\|_{1}\leq\|\mathcal{L}_{A}^{k}(P^{+})\|_{1}+\|\mathcal{L}_{A}^{k}(P^{-})\|_{1}.$
Next, we will show that for any $\bar{P}\in\mathbb{H}_{1}^{n+}$, there are $\beta_{1} \geq 1$, $ \alpha \in(0,1)$ such that
\begin{equation}\label{890barPLEQN}
\|\mathcal{L}_{A}^{k}(\bar{P})\|_{1}\leq \beta_{1}\alpha^{k}\|\bar{P}\|_{1}, \ \forall k\in\mathbb{N}.
\end{equation}
When $\|\bar{P}\|_{1}=0$, \eqref{890barPLEQN} is obviously trivial. Let $\|\bar{P}\|_{1}\neq 0$
and set $\hat{P}=\{\hat{P}(\ell)\}_{\ell\in\Theta}$ with $\hat{P}(\ell)=\|\bar{P}(\ell)\| I_{n\times n}$.
It is obvious that $\|\bar{P}\|_{1}=\|\hat{P}\|_{1}$ and $\bar{P}\leq \hat{P}$.
Hence $0 \leq \mathcal{L}_{A}^{k}(\bar{P}) \leq \mathcal{L}_{A}^{k}(\hat{P})$, which yields that
$\|\mathcal{L}_{A}^{k}(\bar{P})\|_{1} \leq \|\mathcal{L}_{A}^{k}(\hat{P})\|_{1}.$
On the other hand, $\hat{P}$ can be written as $\hat{P}=\sum_{i=1}^{n} \hat{P}_{i}$
with $\hat{P}_{i}=\{\hat{P}_{i}(\ell)\}_{\ell\in\Theta}$,
where $\hat{P}_{i}(\ell)=\|\bar{P}(\ell)\|e_{i}e_{i}'$ with $e_{i}'=(0,0,\cdots,0,1,0,\cdots,0)$
($e_{i}\in \mathbb{R}^{n}$ is the vector where the $i$-th term is 1 and the rest is 0).
Now we arrive at
\begin{equation}\label{913KAHATPII}
\|\mathcal{L}_{A}^{k}(\hat{P})\|_{1}\leq \sum_{i=1}^{n}\|\mathcal{L}_{A}^{k}(\hat{P}_{i})\|_{1},\ \forall k\in\mathbb{N}.
\end{equation}
As for $\hat{P}_{i}$, $i\in \overline{1,n}$, notice that $\|\hat{P}_{i}\|_{1}=\|\bar{P}\|_{1}$.
In addition, there is $\nu\in\mathbb{H}_{1}^{n+}$ given by $\nu(\ell)={\|\bar{P}(\ell)\|}/{\|\bar{P}\|_{1}}$
and $x_{0}^{i}=\sqrt{\|\bar{P}\|_{1}}e_{i}\in\mathbb{R}^{n}$
such that $\hat{P}_{i}=X_{0}^{i}$ with $X_{0}^{i}(\ell)=x_{0}^{i} (x_{0}^{i})'\nu(\ell)$.
So, one can rewrite \eqref{913KAHATPII} as
$\|\mathcal{L}_{A}^{k}(\hat{P})\|_{1}$ $\leq \sum_{i=1}^{n}\|\mathcal{L}_{A}^{k}(X_{0}^{i})\|_{1}.$
Moreover, it can be directly drawn that $\|X_{0}^{i}\|_{1}=\|\hat{P}_{i}\|_{1}=\|\bar{P}\|_{1}$.
Then, according to \eqref{880emss}, for every $i\in\overline{1,n}$, there are $\beta \geq 1$, $\alpha \in(0,1)$ such that
$\|\mathcal{L}_{A}^{k}(X_{0}^{i})\|_{1}\leq n\beta\alpha^{k}\|\bar{P}\|_{1}$
holds for every $k \in \mathbb{N}$, $x_{0}^{i} \in \mathbb{R}^{n}$.
Further, $\|\mathcal{L}_{A}^{k}(\bar{P})\|_{1} \leq n^{2}\beta\alpha^{k}\|\bar{P}\|_{1}$,
i.e.,
\eqref{890barPLEQN} is proved with $\beta_{1}=n^{2}\beta$.
Consequently,
$\|\mathcal{L}_{A}^{k}(P^{+})\|_{1}\leq \beta_{1}\alpha^{k}\|P^{+}\|_{1}\leq \beta_{1}\alpha^{k}\|P\|_{1}$
and $\|\mathcal{L}_{A}^{k}(P^{-})\|_{1}\leq \beta_{1}\alpha^{k}\|P\|_{1},\ \forall k\in\mathbb{N}.$
Now we conclude that there are $\beta_{2}=2 n^{2}\beta \geq 1$, $\alpha \in(0,1)$ such that
$\|\mathcal{L}_{A}^{k}(P)\|_{1}\leq \beta_{2}\alpha^{k}\|P\|_{1}$
holds for every $k \in \mathbb{N}$, $P\in\mathbb{SH}_{1}^{n}$.
Thus, $\mathcal{L}_{A}$ generates an ESCE.
\end{proof}

Noticing the implications that $\mathcal{T}_{A}$ ($\mathcal{L}_{A}$) generates an ESAE (ESCE)
iff $r_{\sigma}(\mathcal{T}_{A})<1$ ($r_{\sigma}(\mathcal{L}_{A})<1$)(see \cite{BookDragan2010}, page 33),
from Theorems \ref{EMSSCiff} and \ref{EMSSCiff2} we can make the following theorem,
which provides relatively complete and unified methods for determining EMSSy-C and EMSSy of $(\mathbf{A};\mathcal{G})$, respectively.
\begin{theorem}\label{spectralcriterion}
$(\mathbf{A};\mathcal{G})$ is EMSS-C iff $r_{\sigma}(\mathcal{T}_{A})<1$;
$(\mathbf{A};\mathcal{G})$ is EMSS iff $r_{\sigma}(\mathcal{L}_{A})<1$.
\end{theorem}
\begin{remark}\label{797reborenoeq}
Consider $(\mathbf{A};\mathcal{G})$.
When specializing Theorem \ref{spectralcriterion} to finite MJLSs, EMSSy-C is equivalent to EMSSy.
Indeed, in this finite case, $\mathcal{T}_{A}$ coincides with the adjoint $\mathcal{L}_{A}^{*}$ of $\mathcal{L}_{A}$.
Hence, $r_{\sigma}(\mathcal{L}_{A})<1$ iff $r_{\sigma}(\mathcal{T}_{A})<1$,
and the equivalence between these two kinds of stabilities of $(\mathbf{A};\mathcal{G})$ can be gained by Theorem \ref{spectralcriterion}.
The same result in the continuous-time case for finite MJLSs could be consulted in Theorem 3.2.4 of \cite{BookDragan2014} (page 135).
However, for the case where the Markov chain takes values in $\Theta$,
establishing the equivalence between these two kinds of stabilities is not a trivial task.
Recalling Proposition \ref{co973EMSSEMSSC},
it shows that $(\mathbf{A};\mathcal{G})$ is EMSS provided $(\mathbf{A};\mathcal{G})$ is EMSS-C.
We conjecture that the converse result still holds under the assumptions considered in this paper.
This problem is challenging but meaningful, and requires a more thorough discussion.
\end{remark}

\begin{conjecture}\label{conjecture}
$(\mathbf{A};\mathcal{G})$ is EMSS iff $(\mathbf{A};\mathcal{G})$ is EMSS-C.
\end{conjecture}

In the following, we will present the equivalent conditions for $(\mathbf{A};\mathcal{G})$ to be EMSS-C
in terms of the existence of uniformly positive definite solutions of coupled Lyapunov-type equations or inequalities in Banach spaces.
\begin{theorem}\label{906theorem31234}
The following assertions are equivalent:\\
(i)~$(\mathbf{A};\mathcal{G})$ is EMSS-C; \\
(ii)~Given any $V\in\mathbb{H}_{\infty}^{n+*}$, there exists $U\in\mathbb{H}_{\infty}^{n+*}$ such that
\begin{equation}\label{Lyapunov902}
U(\ell)-\mathcal{T}_{A}(U)(\ell)=V(\ell)\  \mu\text{-}a.e.;
\end{equation}
(iii)~There exists $U\in\mathbb{H}_{\infty}^{n+*}$ such that
$U(\ell)-\mathcal{T}_{A}(U)(\ell)= I_{n\times n}$ $\mu\text{-}a.e.$;\\
(iv)~There is $U\in\mathbb{H}_{\infty}^{n+*}$ and $\xi>0$ such that
$U(\ell)-\mathcal{T}_{A}(U)(\ell)\geq \xi I_{n\times n}$ $\mu\text{-}a.e.$.
\end{theorem}
\begin{proof}
First, we prove $(i)\rightarrow (ii)$.
It follows from Theorem \ref{spectralcriterion} that $r_{\sigma}(\mathcal{T}_{A})<1$ when $(\mathbf{A};\mathcal{G})$ is EMSS-C.
Hence, for any $V\in\mathbb{H}_{\infty}^{n+*}$, $U=(\mathcal{I}_{\mathbb{SH}_{\infty}^{n}}-\mathcal{T}_{A})^{-1}(V)
=\sum_{k=0}^{\infty}\mathcal{T}_{A}^{k}(V)\in\mathbb{H}_{\infty}^{n+*}$ satisfies \eqref{Lyapunov902}.
The implications $(ii)\rightarrow (iii)\rightarrow (iv)$ can be got from
the standard arguments concerning the bounded linear operator on Banach space $\mathbb{H}_{\infty}^{n}$, and the details are omitted.
Now we prove $(iv)\rightarrow (i)$.
Set $V=U-\mathcal{T}_{A}(U)$.
If $(iv)$ holds, then $\delta I_{n\times n}> U(\ell)\geq \xi I_{n\times n}$ $\mu\text{-}a.e.$, where $\delta>\|U\|_{\infty}$.
Let $\alpha=1-\frac{\xi}{\delta}$. Obviously, $\alpha\in(0,1)$.
Next, for every $k\in\mathbb{N}^{+}$, we will show that
\begin{equation}\label{1074taku}
\mathcal{T}_{A}^{k}(U)(\ell)\leq \alpha^{k}U(\ell)
\end{equation}
holds by mathematical induction.
For $k=1$, since $U(\ell)-\mathcal{T}_{A}(U)(\ell)\geq\xi I_{n\times n}>\frac{\xi}{\delta}U(\ell)$,
it implies that \eqref{1074taku} holds for $k=1$.
We assume that \eqref{1074taku} is true for $k=i\in\mathbb{N}^{+}$.
Considering that $\mathcal{T}_{A}$ is a positive operator on $\mathbb{SH}_{\infty}^{n}$,
thus
$\mathcal{T}_{A}^{i}(U)(\ell)-\mathcal{T}_{A}^{i+1}(U)(\ell)
\geq \xi\mathcal{T}_{A}^{i}(\mathcal{I})(\ell)\geq \frac{\xi}{\delta} \mathcal{T}_{A}^{i}(U)(\ell).$
Combining this with the assumption of \eqref{1074taku} holding for $k=i$,
we get that $\mathcal{T}_{A}^{i+1}(U)(\ell)\leq \alpha\mathcal{T}_{A}^{i}(U)(\ell)
\leq \alpha^{i+1}U(\ell)$.
Therefore, for every $k\in\mathbb{N}^{+}$, \eqref{1074taku} is valid.
So, it can be confirmed that
$\xi \mathcal{T}_{A}^{k}(\mathcal{I})(\ell)\leq \mathcal{T}_{A}^{k}(U)(\ell)\leq \alpha^{k}U(\ell)
\leq \delta \alpha^{k}\mathcal{I}(\ell)$.
This gives rise to $\mathcal{T}_{A}^{k}(\mathcal{I})(\ell)\leq \beta\alpha^{k}I_{n\times n}$,
where $\beta=\frac{\delta}{\xi}>1$.
Further, we have that $\|\mathcal{T}_{A}^{k}(\mathcal{I})(\ell)\|\leq \beta\alpha^{k}$,
and $(i)$ is then inferred according to Corollary \ref{coro742tai} and Theorem \ref{EMSSCiff}.
\end{proof}
\begin{theorem}\label{1365coIF1367EMSSC}
If $(\mathbf{A};\mathcal{G})$ is EMSS-C, then there exists
$U\in\mathbb{H}^{n+}_{\infty}$ such that
\begin{equation}\label{1367utaucc}
U(\ell)-\mathcal{T}_{A}(U)(\ell)=C(\ell)'C(\ell)\ \mu\text{-}a.e..
\end{equation}
\end{theorem}
\begin{proof}
Following the same reasoning as in Theorem \ref{906theorem31234} for the implication $(i)\rightarrow (ii)$, the desired result can be proved.
It is worth noting that the positive semidefinite term $C(\ell)'C(\ell)$ only guarantees that the solution of \eqref{1367utaucc} satisfies $U\in\mathbb{H}_{\infty}^{n+}$.
\end{proof}

\section{Bounded Real Lemma}\label{Bounded Real Lemma}
In this section, we aim to develop the $H_{\infty}$ theory for the perturbation attenuation analysis of $\Phi_{v}$,
where $v$ is an exogenous perturbation with finite energy.
Both the finite and infinite horizon cases are included.
Now we start by introducing the following sets:
For $T\in\mathbb{N}$,
$l^{2}(\overline{0,T};\mathbb{R}^{r})=\{v|v=\{v(k)\}_{k\in\overline{0,T}}$
is an $\mathbb{R}^{r}$-valued random variables sequence
with that $v(k)$ is $\mathfrak{F}_{k}$-measurable, $\forall k\in \overline{0,T}$,
and $\|v\|_{l^{2}(\overline{0,T};\mathbb{R}^{r})}:
=(\sum_{k=0}^{T}\mathbb{E}\{\|v(k)\|^{2}\})^{\frac{1}{2}}<\infty\}$.
For any $\ell\in\Theta$,
$l^{2}_{\ell}(\overline{0,T};\mathbb{R}^{r})
=\{v|v=\{v(k)\}_{k\in\overline{0,T}}$ is an $\mathbb{R}^{r}$-valued random variables sequence
with that $v(k)$ is $\mathfrak{F}_{k}$-measurable, $\forall k\in \overline{0,T}$, and
$\|v\|_{l^{2}_{\ell}(\overline{0,T};\mathbb{R}^{r})}=:(\sum_{k=0}^{T}$
$\mathbb{E}\{\|v(k)\|^{2}|\vartheta_{0}=\ell\})^{\frac{1}{2}}<\infty\}$;
$l^{2}_{\ell}(\mathbb{N};\mathbb{R}^{r})
=\{v|v=\{v(k)\}_{k\in\mathbb{N}}$
is an $\mathbb{R}^{r}$-valued random variables sequence with that $v(k)$ is $\mathfrak{F}_{k}$-measurable, $\forall k\in \mathbb{N}$,
and $\|v\|_{l^{2}_{\ell}(\mathbb{N};\mathbb{R}^{r})}$
$:=(\sum_{k=0}^{\infty}$
$\mathbb{E}\{\|v(k)\|^{2}|\vartheta_{0}=\ell\})^{\frac{1}{2}}<\infty\}$.

\begin{proposition}
(i)~$l^{2}(\overline{0,T};\mathbb{R}^{r})$ is a Hilbert space with the norm
$\|\cdot\|_{l^{2}(\overline{0,T};\mathbb{R}^{r})}$
induced by the usual inner product;\\
(ii)~For any $\ell\in\Theta$, $l^{2}_{\ell}(\overline{0,T};\mathbb{R}^{r})$ and $l^{2}_{\ell}(\mathbb{N};\mathbb{R}^{r})$
are Hilbert spaces with
$\|\cdot\|_{l^{2}_{\ell}(\overline{0,T};\mathbb{R}^{r})}$ and $\|\cdot\|_{l^{2}_{\ell}(\mathbb{N};\mathbb{R}^{r})}$, respectively.
\end{proposition}
\begin{proof}
Let $\widetilde{l}^{2}(\overline{0,T};\mathbb{R}^{r})$ be the space of all $r$-dimensional measurable and square-summable random sequences
$v(\cdot) : \overline{0,T} \times \Omega\rightarrow\mathbb{R}^{r}$.
Following a similar argument as in \cite{BookCosta2005},
it makes known that $\widetilde{l}^{2}(\overline{0,T};\mathbb{R}^{r})$ is a Hilbert space
with the norm $\|v\|_{2}=(\sum_{k=0}^{T}\mathbb{E}\{\|v(k)\|^{2}\})^{\frac{1}{2}}$.
Therefore, $l^{2}(\overline{0,T};\mathbb{R}^{r})$ is a Hilbert space with $\|\cdot\|_{l^{2}(\overline{0,T};\mathbb{R}^{r})}$
owing to being a closed subspace of $\widetilde{l}^{2}(\overline{0,T};\mathbb{R}^{r})$.
For any $\ell\in\Theta$, it can be shown that $\widetilde{l}^{2}_{\ell}(\overline{0,T};\mathbb{R}^{r})$
and $l^{2}_{\ell}(\mathbb{N};\mathbb{R}^{r})$ are also Hilbert spaces similarly.
\end{proof}

For convenience, $v\in l^{2}_{\Theta}(\overline{0,T};\mathbb{R}^{r})$ means that $v\in l^{2}_{\ell}(\overline{0,T};\mathbb{R}^{r})$
and $\|v\|_{l^{2}_{\ell}(\overline{0,T};\mathbb{R}^{r})}<\infty$ $\mu\text{-}a.e.$; $v\in l^{2}_{\Theta}(\mathbb{N};\mathbb{R}^{r})$ means that
$v\in l^{2}_{\ell}(\mathbb{N};\mathbb{R}^{r})$ and $\|v\|_{l^{2}_{\ell}(\mathbb{N};\mathbb{R}^{r})}<\infty$ $\mu\text{-}a.e.$.

In what follows, to avoid confusion, we represent the state and output response of $\Phi_{v}$
as the generalized functions of the initial conditions:
$x(\cdot)=\phi_{x}(\cdot,x_{0},\vartheta_{0},v)$ is the state response of $\Phi_{v}$
with the initial conditions $\vartheta(0)=\vartheta_{0}$ and $x(0)=x_{0}$;
$y(\cdot)=\phi_{y}(\cdot,x_{0},\vartheta_{0},v)=C(\vartheta(\cdot))\phi_{x}(\cdot,x_{0},\vartheta_{0},v)+D(\vartheta(\cdot))v(\cdot)$
is the corresponding output response.
Given $\gamma>0$, for $U\in\mathbb{H}_{\infty}^{n}$, define
\begin{eqnarray}
&&\ \ \ \ \ \ \Psi_{1}(U)(\ell)=\mathcal{T}_{A}(U)(\ell)+C(\ell)'C(\ell),\nonumber\\
&&\ \ \ \ \ \ \Psi_{2}(U)(\ell)=A(\ell)'\mathcal{E}(U)(\ell)B(\ell),
\nonumber\\
&&\ \ \ \ \ \ \Psi_{3}^{\gamma}(U)(\ell)=\mathcal{T}_{B}(U)(\ell)+D(\ell)'D(\ell)-\gamma^{2} I_{r\times r},\nonumber\\
&&\ \ \ \ \ \ \mathcal{F}(U)(\ell)=-\Psi_{3}^{\gamma}(U)(\ell)^{-1}
\Psi_{2}(U)(\ell)', \ \ell\in\Theta, \nonumber
\end{eqnarray}
where $\mathcal{T}_{B}$ is given by \eqref{TK} with $K=B$.

\subsection{The finite horizon case}\label{2028finhcbrl}
In this subsection, the finite horizon BRL, given in Theorem \ref{2192fiboudedlemma},
presents a sufficient and necessary condition to ensure that the finite horizon $H_{\infty}$ norm of the linear perturbation operator
is below a prescribed level $\gamma>0$.

Consider $\Phi_{v}$.
For $T\in\mathbb{N}$, if $v\in l^{2}(\overline{0,T};\mathbb{R}^{r})$, then
$\{\phi_{y}(k,0,\vartheta_{0},v)\}_{k\in\overline{0,T}}\in l^{2}(\overline{0,T};\mathbb{R}^{m})$.
Hence, we could define the linear perturbation operator
$\mathbb{L}_{T}:l^{2}(\overline{0,T};\mathbb{R}^{r})\rightarrow l^{2}(\overline{0,T};\mathbb{R}^{m})$
with $\mathbb{L}_{T}(v)(k)=\phi_{y}(k,0,\vartheta_{0},v),\ \forall k\in\overline{0,T},$
and its $H_{\infty}$ norm $\|\mathbb{L}_{T}\|$ is
\begin{equation}\label{1917tLTnorm}
\|\mathbb{L}_{T}\|:=\sup_{\substack{v\in l^{2}(\overline{0,T};\mathbb{R}^{r}),\\ v\neq0, \ x_{0}=0}}\frac{\|\mathbb{L}_{T}(v)\|_{l^{2}(\overline{0,T};\mathbb{R}^{m})}}
{\|v\|_{l^{2}(\overline{0,T};\mathbb{R}^{r})}}.
\end{equation}
$\|\mathbb{L}_{T}\|$ can measure the effect of the unknown disturbance $v$ on the output $\{\phi_{y}(k,0,\vartheta_{0},v)\}_{k\in\overline{0,T}}$
in the worst-case scenario,
and a larger value of $\|\mathbb{L}_{T}\|$ indicates a greater effect of $v$ on $\{\phi_{y}(k,0,\vartheta_{0},v)\}_{k\in\overline{0,T}}$.
Therefore, finding a way to make $\|\mathbb{L}_{T}\|$ smaller than a prescribed level is an important issue
in analyzing the disturbance attenuation properties of $\Phi_{v}$.

Given $\gamma>0$ and $T\in\mathbb{N}$, for any $x_{0}\in\mathbb{R}^{n}$, define the cost functionals
\begin{eqnarray}
&&J_{\gamma}(T,x_{0},v):=\sum_{k=0}^{T}\mathbb{E}
\{\|y(k)\|^{2}-\gamma^{2}\|v(k)\|^{2}\},\nonumber\\
&&\ \ \ \ \ \ \ \ \ \ \ \ \ \ \ \ \ \ \ \ \ \ \ \ \ \ \ \ \ \ \ \
\forall v\in l^{2}(\overline{0,T};\mathbb{R}^{r});\nonumber\\
&&J_{\gamma}(T,x_{0},\ell,v):=\sum_{k=0}^{T}\mathbb{E}
\{\|y(k)\|^{2}-\gamma^{2}\|v(k)\|^{2}|\vartheta_{0}=\ell\},\nonumber\\
&&\ \ \ \ \ \ \ \ \ \ \ \ \ \ \ \ \ \ \ \ \ \ \ \ \ \ \ \ \ \ \ \
\forall v\in l^{2}_{\ell}(\overline{0,T};\mathbb{R}^{r}),\ \forall \ell\in\Theta.\label{1811costf1}
\end{eqnarray}

The following propositions state some properties of $\|\mathbb{L}_{T}\|$ and the cost functionals.
The proofs are relatively simple and therefore omitted.
\begin{proposition}\label{1836proJTXOVINT}
Given $\gamma>0$ and $T\in\mathbb{N}$, for any $v\in l^{2}_{\Theta}(\overline{0,T};\mathbb{R}^{r})$, $J_{\gamma}(T,x_{0},v)=\int_{\Theta}J_{\gamma}(T,x_{0},\ell,v)\nu(\ell)\mu(d\ell)$.
\end{proposition}
\begin{proposition}\label{pro1821LJJJ}
Given $\gamma>0$ and $T\in\mathbb{N}$, for any $v\in l^{2}(\overline{0,T};\mathbb{R}^{r})$
with $v\neq 0$,
$J_{\gamma}(T,0,v)\leq0$ iff $\|\mathbb{L}_{T}\|\leq\gamma$.
Moreover,
$J_{\gamma}(T,0,v)<0$ iff $\|\mathbb{L}_{T}\|<\gamma$.
\end{proposition}

Given $\gamma>0$ and $T\in\mathbb{N}$, for every $k\in\overline{0,T}$, set $F(k)=\{F(k,\ell)\}_{\ell\in\Theta}\in \mathbb{H}_{\infty}^{r\times n}$.
Consider the backward difference equation:
\begin{equation}\label{1863YA+BF}
\left\{
\begin{array}{lr}
Y (k,\ell )=[A (\ell )+B (\ell )F (k,\ell )]'\mathcal{E} (Y (k+1 ) ) (\ell )\\
\ \ \ \ \ \ \cdot[A (\ell )+B (\ell )F (k,\ell )]-\gamma^{2}F (k,\ell )'F (k,\ell )\\
\ \ \ \ \ \ \!+\![C(\ell)\!+\!D(\ell)F(k,\ell)]'[C(\ell )\!+\!D(\ell )F(k,\ell)], \\
Y (T+1,\ell )=0, \ \forall k\in\overline{0,T},\ \ell\in\Theta.
\end{array}
\right.
\end{equation}
Obviously, the solution of \eqref{1863YA+BF} can be obtained by iteration and is denoted by
$Y_{T}^{F,\gamma}(k)=\{Y_{T}^{F,\gamma}(k,\ell)\}_{\ell\in\Theta}\in\mathbb{H}_{\infty}^{n},\ k\in\overline{0,T+1}.$
In addition, we can rewrite \eqref{1863YA+BF} as
\begin{equation}\label{1882eq}
\left(
                       \begin{array}{c}
                         I_{n\times n} \\
                         F\left(k,\ell\right)\\
                       \end{array}
                     \right)'
W\left(Y_{T}^{F,\gamma},k,\ell\right)
\left(
  \begin{array}{c}
    I_{n\times n} \\
    F\left(k,\ell\right) \\
  \end{array}
\right)=0,
\end{equation}
where\\
$W\!\left(\!Y,k,\ell\right)
\!=\!\left(\!
  \begin{array}{cc}
    \!\Psi_{\!1\!}(Y\!(k\!+\!1))(\ell)\!-\!Y\!(k)(\ell)\! & \!\Psi_{\!2\!}(Y\!(k\!+\!1))(\ell)\! \\
    \!\Psi_{\!2\!}(Y\!(k\!+\!1))(\ell)' & \!\Psi_{\!3\!}^{\!\gamma\!}(Y
    (k\!+\!1))(\ell)\! \\
  \end{array}
\right).$

The following lemma will be used to assist in proving the necessity part of the finite horizon BRL (Theorem \ref{2192fiboudedlemma}).
\begin{lemma}\label{1861nepro}
Fix $\gamma>0$ and $T\in\mathbb{N}$.
Set $F(k)\in \mathbb{H}_{\infty}^{r\times n}$, $k\in\overline{0,T}$.
If $\Phi_{v}$ satisfies $\|\mathbb{L}_{T}\|<\gamma$, then \eqref{1863YA+BF} admits a solution such that
\begin{equation}\label{1863props3}
\Psi_{3}^{\gamma}(Y_{T}^{F,\gamma}(k\!+\!1))(\ell)\leq-\eta_{0} I_{r\times r}\ \mu\text{-}a.e., \ \forall\eta_{0}\!\in\!(0,\gamma^{2}\!-\!\|\mathbb{L}_{T}\|^{2})
\end{equation}
holds for every $k\in\overline{0,T}$.
\end{lemma}
\begin{proof}
The proof of this lemma will be divided into two steps.
First, we will prove that
\begin{equation}\label{1892step1}
\Psi_{3}^{\gamma}(Y_{T}^{F,\gamma}(k+1))\left(\ell\right)\leq0\ \mu\text{-}a.e.,
\forall k\in\overline{0,T}.
\end{equation}
The second step, following a similar argument as the proof of Proposition 8.4 in \cite{BookDragan2010},
can use \eqref{1892step1} to show \eqref{1863props3}.
Therefore, here we just have to prove \eqref{1892step1}.
Since $\Phi_{v}$ satisfies $\|\mathbb{L}_{T}\|<\gamma$, by Proposition \ref{pro1821LJJJ},
for any $v\in l^{2}(\overline{0,T};\mathbb{R}^{r})$ with $v\neq 0$, we have that $J_{\gamma}(T,0,v)<0$.
Now let us assume that \eqref{1892step1} does not hold,
then there is $k_{0}\in\overline{0,T}$ and $\Delta\in\mathcal{B}(\Theta)$ with $\mu(\Delta)>0$
such that for any $\ell_{\Delta}\in\Delta$, there exists $\kappa(\ell_{\Delta})\in \mathbb{R}^{r}$
satisfying $\|\kappa(\ell_{\Delta})\|=1$ and that
\begin{equation}\label{1899fou}
\kappa(\ell_{\Delta})'
\Psi_{3}^{\gamma}(Y_{T}^{F,\gamma}(k_{0}+1))
\left(\ell_{\Delta}\right)
\kappa(\ell_{\Delta})>0
\end{equation}
holds $\mu$-$a.e.$ on $\Delta$.
For every $k\in\overline{0,T}$,
define $\bar{v}(k)\in l^{2}(\overline{0,T};\mathbb{R}^{r})$ by
$\bar{v}(k)\!=\!\left\{
\begin{aligned}
\kappa(\vartheta(k)) \cdot \chi_{\{\vartheta(k)\in \Delta\}},& \ k\!=\!k_{0},  \\
0,\ \ \ \ \ & \ k\!\neq\! k_{0},
\end{aligned} \right.$
in which $\chi_{\{\cdot\}}$ is an indicator function.
Set $\hat{v}(k)=F(k,\vartheta(k))z(k)+\bar{v}(k)$, $k\in\overline{0,T}$, where $z(k)=\phi_{x}(k,0,\vartheta_{0},\hat{v})$.
Clearly, $\hat{v}\in l^{2}(\overline{0,T};\mathbb{R}^{r})$ and $\hat{v}\neq 0$.
Noticing that $Y_{T}^{F,\gamma}(k),\ k\in\overline{0,T+1}$ is the solution of \eqref{1863YA+BF},
we calculate $\mathbb{E}\{[z(k+1)'Y^{F,\gamma}_{T}(k+1)(\vartheta(k+1))z(k+1) -z(k)'Y^{F,\gamma}_{T}(k)(\vartheta(k))z(k)]|\vartheta_{0}=\ell\}$
and sum it over $k$, then
\begin{equation}\label{2010sumo0t}
\begin{split}
0=&\sum_{k=0}^{T}\mathbb{E}\{[\left(
                       \begin{array}{c}
                         z(k) \\
                         \hat{v}(k) \\
                       \end{array}
                     \right)'
W\left(Y^{F,\gamma}_{T},k,\vartheta(k)\right)
              \left(
                       \begin{array}{c}
                         z(k) \\
                         \hat{v}(k) \\
                       \end{array}
                    \right)\\
&\ \ \ \ \ \ \ \ \ \ \!-\!\hat{v}(k)'(D(\vartheta(k))'\!D(\vartheta(k))
\!-\!\gamma^{2}\! I_{r\times r})\hat{v}(k)\\
&\ \ \ \ \ \ \ \ \ \ \!-z(k)'C(\vartheta(k))'C(\vartheta(k))z(k)]|\vartheta_{0}\!=\!\ell\}.
\end{split}
\end{equation}
Considering \eqref{1882eq} and \eqref{2010sumo0t}, one can rewrite \eqref{1811costf1} for $\Phi_{\hat{v}}$ with $x_{0}=0$ as follows:
\begin{equation}\label{2193J2+VV}
\begin{split}
&J_{\gamma}(T,0,\ell,\hat{v})\\
=&\sum_{k=0}^{T}\mathbb{E}\{\bar{v}(k)'\Psi_{3}^{\gamma}(Y_{T}^{F,\gamma}(k+1))
(\vartheta(k))\bar{v}(k)\\
&\ \ \ \ \ \  +2z(k)'[ \Psi_{3}^{\gamma}(Y_{T}^{F,\gamma}(k+1))(\vartheta(k))F(k,\vartheta(k))\\
&\ \ \ \ \ \  +\Psi_{2}(Y_{T}^{F,\gamma}(k+1))(\vartheta(k))]\bar{v}(k)|\vartheta_{0}=\ell\}.
\end{split}
\end{equation}
Since $\bar{v}(k)=0$ for $k\neq k_{0}$ and $z(k)$ is the state response of $\Phi_{\hat{v}}$ with $x_{0}=0$,
$z(k)=0$ for $k\leq k_{0}$ is concluded.
So, \eqref{2193J2+VV} can be rewritten as
$J_{\gamma}(T,0,\ell,\hat{v})\!=\!\mathbb{E}\{\bar{v}(k_{0})'\Psi_{3}^{\gamma}(Y_{T}^{F,\gamma}(k_{0}+1))
(\vartheta(k_{0}))\bar{v}(k_{0})|\vartheta_{0}=\ell\}.$
In view of this, by Proposition \ref{1836proJTXOVINT}, we have that
$J_{\gamma}(T,0,\hat{v})=\int_{\!\Theta\!}$
$\mathbb{E}\{\kappa(\vartheta(k_{0}))'
\Psi_{3}^{\gamma}(Y_{T}^{F,\gamma}(k_{0}\!+\!1))
(\vartheta(k_{0}))
\kappa(\vartheta(k_{0}))$
$\cdot\chi_{\{\vartheta(k_{0})\in \Delta\}}$
$|\vartheta_{0}\!=\!\ell\}\nu(\ell)\mu(d\ell).$
Further, $J_{\gamma}(T,0,\hat{v})> 0$
follows from \eqref{1899fou} and Assumptions \ref{Ass2902} and \ref{1977assgtl0},
which is a contradiction with that for any $v\in l^{2}(\overline{0,T};\mathbb{R}^{r})$ and $v\neq 0$, $J_{\gamma}(T,0,v)<0$.
\eqref{1892step1} is therefore proved.
\end{proof}
\begin{remark}
When the Markov chain takes values in a countable set, to prove \eqref{1892step1},
it only needs to assume the existence of a discrete point $l_{0}\in\Theta$ satisfying \eqref{1899fou}
to lead to a contradiction (refer to the proof of Proposition 8.4 in \cite{BookDragan2010}).
However, as mentioned earlier, the Borel space considered in this paper is a general state space,
including the case of continuous state space, so the above method is no longer applicable,
and it requires to utilize measure theory to derive Lemma \ref{1861nepro}.
\end{remark}

The result below is immediate from Lemma \ref{1861nepro}.
\begin{proposition}
Given $\gamma>0$, if there exists $T\in\mathbb{N}^{+}$ such that $\Phi_{v}$ satisfies $\|\mathbb{L}_{T}\|<\gamma$,
then $D(\ell)'D(\ell)-\gamma^{2} I_{r\times r}<0$ $\mu\text{-}a.e.$.
\end{proposition}

Now we are prepared to propose the finite horizon BRL, which is the main result of this subsection.
\begin{theorem}\label{2192fiboudedlemma}
Given $\gamma>0$ and $T\in\mathbb{N}$, $\Phi_{v}$ satisfies $\|\mathbb{L}_{T}\|<\gamma$ iff the following backward difference equation:
\begin{equation}\label{1785ARE}
\left\{
\begin{array}{ll}
Y(k,\ell)\!=&\!\Psi_{1}(Y(k\!+\!1))(\ell)-\Psi_{2}(Y(k\!+\!1))(\ell)\\
&\cdot\Psi_{3}^{\gamma}(Y(k\!+\!1))(\ell)^{-1}\Psi_{2}(Y(k\!+\!1))(\ell)',\\
Y(T+1,&\ell)=0,\ \forall k\in\overline{0,T},\ \ell\in\Theta\\
\end{array}
\right.
\end{equation}
admits a solution $Y_{T}^{\gamma}(k)\in\mathbb{H}_{\infty}^{n+}$ such that
\begin{equation}\label{1794ineqxt}
\Psi_{3}^{\gamma}\!\left(Y_{T}^{\gamma}\!
\left(k\!+\!1\right)\right)\left(\ell\right)\!\leq\!-\!\eta_{0} I_{r\times r}\ \mu\text{-}a.e., \forall\eta_{0}\!\in\!\left(0,\!\gamma^{2}\!-\!\|\mathbb{L}_{T}\|^{2}\right),
\end{equation}
for every $k\in\overline{0,T}$.
\end{theorem}
\begin{proof}
Sufficiency.
If \eqref{1785ARE} admits a solution $Y_{T}^{\gamma}(k)\in\mathbb{H}_{\infty}^{n+}$, $k\in\overline{0,T}$, then
\begin{eqnarray}\label{Jx01871}
J_{\gamma}(T,x_{0},v)\!
=\!&&\sum_{k=0}^{T}\!\mathbb{E}\!\left\{\!\left(\!
                            \begin{array}{c}
                                            x(k) \\
                                            v(k) \\
                                          \end{array}
                                        \!\right)'\!
                            W\!\left(Y_{T}^{\gamma},k,\vartheta\left(k\right)\right)
                            \!\left(\!
                            \begin{array}{c}
                                            x(k) \\
                                            v(k) \\
                                          \end{array}
                                       \! \right)\!
                            \right\}\nonumber \\
&&+\mathbb{E}\left\{x_{0}'Y_{T}^{\gamma}\left(\vartheta_{0}\right)x_{0}\right\},
\end{eqnarray}
where $x(\cdot)$ is the system state of $\Phi_{v}$.
Since $Y_{T}^{\gamma}(k)$, $k\in\overline{0,T}$ is the solution of \eqref{1785ARE}, then by letting $x_{0}=0$,
\eqref{Jx01871} can be rewritten as
$J_{\gamma}(T,0,v)=\sum_{k=0}^{T}\mathbb{E}\{[v(k)-\hat{v}(k)]'$ $\cdot\Psi_{3}^{\gamma}(Y_{T}^{\gamma}(k+1))(\vartheta(k))
[v(k)-\hat{v}(k)]\}\leq 0$,
in which $\hat{v}(k)=\mathcal{F}(Y_{T}^{\gamma}(k+1))(\vartheta(k))\phi_{x}(k,0,\vartheta_{0},v)$.
Combining this with \eqref{1794ineqxt}, one obtains that $J_{\gamma}(T,0,v)=0$ iff $v(k)=\hat{v}(k)$, $\forall k\in\overline{0,T}$.
Now substituting $\hat{v}(k)$ into $\Phi_{v}$ with $x_{0}=0$, we deduce that
$\phi_{x}(k,0,\vartheta_{0},\hat{v})\equiv 0$, which implies that for every $k\in\overline{0,T}$, $\hat{v}(k)=0$.
Hence, if $v(k)\neq0$, then
$J_{\gamma}(T,0,v)\!\leq\!-\eta_{0}\|v-\hat{v}\|^{2}
_{l^{2}(\overline{0,T};\mathbb{R}^{r})}<0.$
Therefore, $\|\mathbb{L}_{T}\|<\gamma$ is concluded by Proposition \ref{pro1821LJJJ}.\\
Necessity.
For $k=T$, by substituting $F(k,\ell)=\mathcal{F}(Y_{T}^{\gamma}(k+1))(\ell)$ into \eqref{1863YA+BF},
we can get that the obtained equation is consistent with \eqref{1785ARE}.
Moreover, $Y_{T}^{\gamma}(T, \ell)=C(\ell)'C(\ell)\geq 0$.
Using the assumption of $\left\|\mathbb{L}_{T}\right\|<\gamma$ and Lemma \ref{1861nepro}, it follows that
\eqref{1794ineqxt} is satisfied for $k=T$.
Next, \eqref{1785ARE} is iteratively proceeded backwards, starting from $T-1$ and stopping at $k=0$.
For each $k\in\{T-1,T-2,\cdots,0\}$, repeating the process as the case of $k=T$,
we assert that \eqref{1785ARE} admits a solution $Y_{T}^{\gamma}(k)\in\mathbb{H}_{\infty}^{n+}$ satisfying \eqref{1794ineqxt}.
\end{proof}

\subsection{The infinite horizon case}

In this subsection, we will give the infinite horizon BRL,
which can characterize internal stability and the $H_{\infty}$ performance of MJLSs
through the existence of the stabilizing solution of the $\Theta$-coupled ARE.
Before that, we provide the following definition,
which generalizes the notion of input-to-state stability to the system considered in this paper:
\begin{definition}
We call that $\Phi_{v}$ is input-to-state stable with conditioning, if for any $x_{0}\in\mathbb{R}^{n}$,
$x=\{x(k)\}_{k\in\mathbb{N}}\in l^{2}_{\Theta}(\mathbb{N};\mathbb{R}^{r})$ whenever $v\in l^{2}_{\Theta}(\mathbb{N};\mathbb{R}^{r})$,
where $x(\cdot)$ is the system state of $\Phi_{v}$.
\end{definition}

For the MJLS where the Markov chain takes values in a countably infinite set,
it was shown in \cite{Hou2016AC} that EMSSy-C of the autonomous system can guarantee the input-to-state stability of $\Phi_{v}$,
which is not achievable with stochastic stability.
The following proposition can be regarded as a generalization of Theorem 3.3 in \cite{Hou2016AC}.
Its proof is similar to that of Theorem 3.3 and thus is omitted.
\begin{proposition}\label{proemsscintst1447}
If $(\mathbf{A};\mathcal{G})$ is EMSS-C, then for any $\ell\in\Theta$,
$x=\{x(k)\}_{k\in\mathbb{N}}\in l^{2}_{\ell}(\mathbb{N};\mathbb{R}^{n})$ whenever $v\in l^{2}_{\ell}(\mathbb{N};\mathbb{R}^{r})$,
where $x(\cdot)$ is the system state of $\Phi_{v}$.
Furthermore, $\Phi_{v}$ is input-to-state stable with conditioning provided $(\mathbf{A};\mathcal{G})$ is EMSS-C.
\end{proposition}

Note that Proposition \ref{proemsscintst1447} yields the next corollary:
\begin{corollary}\label{co2363lims}
If $(\mathbf{A};\mathcal{G})$ is EMSS-C,
then we have that $\lim_{k\rightarrow\infty}\mathbb{E}\{\|x(k)\|^{2}|\vartheta_{0}=\ell\}=0$ $\mu\text{-}a.e.$
whenever $\sum_{k=0}^{\infty} \mathbb{E}\{\|v(k)\|^{2}|\vartheta_{0}=\ell\}<\infty$ $\mu\text{-}a.e.$,
where $x(\cdot)$ is the system state of $\Phi_{v}$.
\end{corollary}

We need to emphasize that all the analysis below is under the assumption that $\Phi_{v}$ is internally stable,
i.e., $(\mathbf{A};\mathcal{G})$ is EMSS-C.

According to Proposition \ref{proemsscintst1447}, for any $\ell\in\Theta$, if $v\in l^{2}_{\ell}(\mathbb{N};\mathbb{R}^{r})$,
then $\{\phi_{x}(k,0,\ell,v)\}_{k\in\mathbb{N}}\in l^{2}_{\ell}(\mathbb{N};\mathbb{R}^{n})$
and a straightforward calculation yields $\{\phi_{y}(k,0,\ell,v)\}_{k\in\mathbb{N}}\in l^{2}_{\ell}(\mathbb{N};\mathbb{R}^{m})$.
Therefore, for every $k\in\mathbb{N}$, we can define the linear perturbation operator
$\mathfrak{L}_{\infty}(v)(k,\ell):=\phi_{y}(k,0,\ell,v)$
and its $H_{\infty}$ norm:
\begin{equation}\label{2359Linfnorm}
\begin{aligned}
\|\mathfrak{L}_{\infty}\|:&
=\esssup_{\substack{\ell\in\Theta,\ v\in l^{2}_{\ell}(\mathbb{N};\mathbb{R}^{r}),\\ v\neq0,\ x_{0}=0}}
\frac{\|\mathfrak{L}_{\infty}(v)\|_{l^{2}_{\ell}(\mathbb{N};\mathbb{R}^{m})}}
{\|v\|_{l^{2}_{\ell}(\mathbb{N};\mathbb{R}^{r})}}.
\end{aligned}
\end{equation}
\begin{remark}
Assumption 2 is necessary to ensure that $\|v\|_{l^{2}_{\ell}(\mathbb{N};\mathbb{R}^{r})}
=(\sum_{k=0}^{\infty}\mathbb{E}\{\|v(k)\|^{2}|\vartheta_{0}=\ell\})^{\frac{1}{2}}>0$  $\mu$-$a.e.$ on $\Theta$
and \eqref{2359Linfnorm} is well-defined.
Indeed, if Assumption 2 does not hold, then there exists $\Delta_{0}\in\mathcal{B}(\Theta)$ with $\mu(\Delta_{0})>0$ such that $\hat{\mu}(\vartheta_{0}\in\Delta_{0})=\int_{\Delta_{0}}\nu(\ell)\mu(d\ell)=0$,
where $\hat{\mu}$ is the distribution of $\vartheta_{0}$.
This implies that $\sum_{k=0}^{\infty}\mathbb{E}\{\|v(k)\|^{2}|\vartheta_{0}=\ell\}=0$ $\mu$-$a.e.$ on $\Delta_{0}$.
One can refer to Remark 8 in \cite{Marcos2018auto} for the case of finite MJLSs.
\end{remark}

Given $\gamma>0$, and letting $T\rightarrow\infty$ in \eqref{1811costf1},
for any $v\in l^{2}_{\Theta}(\mathbb{N};\mathbb{R}^{r})$,
we have the corresponding form for the infinite horizon case denoted by  $J_{\gamma}(\infty,x_{0},\ell,v)$.
Similar to the scenario of the finite horizon,
the following proposition on $\|\mathfrak{L}_{\infty}\|$ and $J_{\gamma}(\infty,x_{0},\ell,v)$ can be shown.
\begin{proposition}\label{pro2394LJJJ}
Given $\gamma>0$, for any $v\in l^{2}_{\Theta}(\mathbb{N};\mathbb{R}^{r})$ with $v\neq 0$, $J_{\gamma}(\infty,0,\ell,v)\leq0$ $\mu\text{-}a.e.$
iff $\|\mathfrak{L}_{\infty}\|\leq\gamma$.
Moreover, $J_{\gamma}(\infty,0,\ell,v)<0$ $\mu\text{-}a.e.$ iff $\|\mathfrak{L}_{\infty}\|<\gamma$.
\end{proposition}

The following job is to apply the finite horizon BRL developed in Subsection \ref{2028finhcbrl} to work out the infinite horizon scenario.
First, we present two necessary propositions.
\begin{proposition}\label{2504<0lTlinf}
For every $T\in\mathbb{N}$, $\|\mathbb{L}_{T}\|\leq\|\mathfrak{L}_{\infty}\|$.
\end{proposition}
\begin{proof}
For every $T\in\mathbb{N}$ and $\ell\in\Theta$, if $v\in l^{2}_{\ell}(\overline{0,T};\mathbb{R}^{r})$,
then $\{\phi_{y}(k,0,\ell,v)\}_{k\in \overline{0,T}}\in l^{2}_{\ell}(\overline{0,T};\mathbb{R}^{m})$.
Therefore, we can give the linear perturbation operator
$\mathfrak{L}_{T}(v)(k,\ell)=\phi_{y}(k,0,\ell,v),\ \forall k\in\overline{0,T}$
and its  $H_{\infty}$ norm:
\begin{equation*}\label{2575LinfTnorm}
\begin{aligned}
\|\mathfrak{L}_{T}\|:&
=\esssup_{\substack{\ell\in\Theta,\   v\in l^{2}_{\ell}(\overline{0,T};\mathbb{R}^{r}),\\
v\neq0,\ x_{0}=0}}
\frac{\|\mathfrak{L}_{T}(v)\|_{l^{2}_{\ell}
(\overline{0,T};\mathbb{R}^{m})}}
{\|v\|_{l^{2}_{\ell}(\overline{0,T};\mathbb{R}^{r})}}.
\end{aligned}
\end{equation*}
Hence, for any $v\in l^{2}_{\Theta}(\overline{0,T};\mathbb{R}^{r})$,
$\|\mathfrak{L}_{T}\|\cdot \sum_{k=0}^{T}\mathbb{E}\{\|v(k)\|^{2}$
$|\vartheta_{0}=\ell\}\geq\sum_{k=0}^{T}
\mathbb{E}\{\|\phi_{y}(k,0,\vartheta_{0},v)\|^{2}|\vartheta_{0}=\ell\}$ $\mu\text{-}a.e.$.
According to Assumption \ref{Ass2902} and Theorem 15.2 in \cite{BookBillingsley1995},
it follows that for any $v\in l^{2}(\overline{0,T};\mathbb{R}^{r})$,
$\sum_{k=0}^{T}\mathbb{E}\{\|\phi_{y}(k,0,\vartheta_{0},v)\|^{2}\}\!\leq\!
\|\mathfrak{L}_{T}\|\cdot \sum_{k=0}^{T}\!\mathbb{E}\{\|v(k)\|^{2}\}$.
Further, bearing in mind \eqref{1917tLTnorm}, we have that $\|\mathbb{L}_{T}\|\leq\|\mathfrak{L}_{T}\|$.
Next, we will prove $\|\mathfrak{L}_{T}\|\leq\|\mathfrak{L}_{\infty}\|$.
For any $v\in l^{2}_{\ell}(\overline{0,T};\mathbb{R}^{r})$, $\ell\in\Theta$,
setting $\bar{v}=\{\bar{v}(k)\}_{k\in\mathbb{N}}\in l^{2}_{\ell}(\mathbb{N};\mathbb{R}^{r})$ with
$\bar{v}(k)=\left\{
\begin{aligned}
v(k),& \ \forall k\in \overline{0,T},  \\
0,\ \ &  otherwise,
\end{aligned} \right.$
one can get that
\begin{equation*}
\begin{split}
&\frac{\sum_{k=0}^{T}
\mathbb{E}\{\|\phi_{y}(k,0,\vartheta_{0},v)\|^{2}
|\vartheta_{0}=\ell\}}{\sum_{k=0}^{T}
\mathbb{E}\{\|v(k)\|^{2}
|\vartheta_{0}=\ell\}}\\
=&\frac{\sum_{k=0}^{\infty}
\mathbb{E}\{\|\phi_{y}(k,0,\vartheta_{0},\bar{v})\|^{2}
|\vartheta_{0}=\ell\}}{\sum_{k=0}^{\infty}
\mathbb{E}\{\|\bar{v}(k)\|^{2}
|\vartheta_{0}=\ell\}}\leq
\|\mathfrak{L}_{\infty}\|,
\end{split}
\end{equation*}
which leads to $\|\mathfrak{L}_{T}\|\leq\|\mathfrak{L}_{\infty}\|$.
Therefore, $\|\mathbb{L}_{T}\|\leq\|\mathfrak{L}_{\infty}\|$ is carried out from $\|\mathbb{L}_{T}\|\leq\|\mathfrak{L}_{T}\|$.
\end{proof}
\begin{proposition}\label{2985jaccle}
Given $\gamma>0$,
for any $x_{0}\in\mathbb{R}^{n}$ and $v\in l^{2}_{\Theta}(\mathbb{N};\mathbb{R}^{r})$, if $\Phi_{v}$ is internally stable and satisfies $\|\mathfrak{L}_{\infty}\|<\gamma$,
then
there exists
$\zeta>0$ such that $J_{\gamma}\left(\infty,x_{0},\ell, v\right) \leq \zeta\left\|x_{0}\right\|^{2}$ $\mu\text{-}a.e.$.
\end{proposition}
\begin{proof}
Let $T\in\mathbb{N}$ be arbitrary but fixed.
If $\Phi_{v}$ is internally stable,
then by Theorem \ref{1365coIF1367EMSSC},
there is $U\in\mathbb{H}^{n+}_{\infty}$ satisfying $U\!-\!\mathcal{T}_{A}(U)\!=\!C'C.$
This leads to
\begin{eqnarray}\label{2404suxtutx-x0}
\begin{split}
&\mathbb{E}\{x(T\!+\!1)'U(\vartheta(T\!+\!1))
x(T\!+\!1)|\vartheta_{0}=\ell\}\!\\
=&x_{0}'U(\ell)x_{0}+\sum_{k=0}^{T}\mathbb{E}\{
-x(k)'C(\vartheta(k))'C(\vartheta(k))x(k)\\
+\!&2x(k)'\Psi_{2}(\!U\!)(\vartheta(k))v(k) \!+\!v(k)'\mathcal{T}_{B}(\!U\!)(\vartheta(k))v(k)
|\vartheta_{0}\!=\!\ell\},
\end{split}
\end{eqnarray}
where $x(\cdot)$ is given by $\Phi_{v}$ with $x(0)=x_{0}$.
For any $v\in l^{2}_{\Theta}(\mathbb{N};\mathbb{R}^{r})$, according to Corollary \ref{co2363lims}, one has that
$\lim_{T\rightarrow\infty}\mathbb{E}\{x(T)'U(\vartheta(T))x(T)|\vartheta_{0}\!=\!\ell\}\!=\!0.$
Letting $T\!\rightarrow\!\infty$ in \eqref{2404suxtutx-x0}, it follows that
$J_{\gamma}\left(\infty,\!x_{0},\!\ell,\! v\right)
\!=\!\sum_{k=0}^{\infty}\mathbb{E}\!\{[2x(k)'
\Psi_{2}(U)(\vartheta(\!k\!))v(k)\!+\!v(k)'\Psi_{3}^{\gamma}(U)(\vartheta(k))v(k)]$ $|\vartheta_{0}=\ell\}
\!+\!x_{0}'U(\ell)x_{0}.$
Then,
$J_{\gamma}\left(\infty,x_{0},\ell, v\right)-J_{\gamma}\left(\infty,0,\ell, v\right)$
$\!=\!x_{0}'U(\ell)x_{0}
+\!\sum_{k=0}^{\infty}\!
\mathbb{E}\{2\bar{x}(k)'
\Psi_{2}(U)(\vartheta(k)) v(k)|\vartheta_{0}\!=\!\ell\},$
where $\bar{x}(k)$ is the system state of $(\mathbf{A};\mathcal{G})$.
Considering that $\|\mathfrak{L}_{\infty}\|<\gamma$,
and taking $\varepsilon$ with $\|\mathfrak{L}_{\infty}\|^{2}\leq\gamma^{2}-\varepsilon^{2}$,
$J_{\gamma}\left(\infty,0,\ell,v\right)\leq-\varepsilon^{2}\sum_{k=0}^{\infty}\mathbb{E}\{\|v(k)\|^{2}|\vartheta_{0}=\ell\}$
can be concluded from Proposition \ref{pro2394LJJJ}.
So,
$J_{\gamma}(\infty,x_{0},\ell, v)
\leq \frac{1}{\varepsilon^{2}}
\sum_{k=0}^{\infty}
\mathbb{E}\{\|\Psi_{2}(U)(\vartheta(k))
\bar{x}(k)\|^{2}
|\vartheta_{0}\!=\!\ell\}$$+x_{0}'U(\ell)x_{0}+\Pi,$
where $\Pi\!=\!-\!\sum_{k=0}^{\infty}\!\mathbb{E}\{\|\varepsilon v(k)\!-\!\frac{1}{\varepsilon}\Psi_{2}(U)(\vartheta(k))
\bar{x}(k)\|^{2}|\vartheta_{0}=\ell\}\leq0.$
Since $(\mathbf{A};\mathcal{G})$ is EMSS-C, there are $\beta \geq 1$, $ \alpha \in(0,1)$ such that
$\sum_{k=0}^{\infty}\mathbb{E}\{\|\bar{x}(k)\|^{2}|$ $\vartheta_{0}=\ell\}\leq   \frac{\beta}{1-\alpha}\|x_{0}\|^{2}.$
Setting $\zeta$ with $\zeta>\|U\|_{\infty}+\frac{\beta}{\varepsilon^{2}(1-\alpha)}\|U\|_{\infty}^{2}
\|A\|_{\infty}^{2}\|B\|_{\infty}^{2}\geq0$,
$J_{\gamma}\left(\infty,x_{0},\ell, v\right)\leq \zeta\|x_{0}\|^{2}$ is then fulfilled.
\end{proof}

Given $\gamma>0$, for every $T\in\mathbb{N}$ and $k\in \overline{0,T+1}$, define
$K_{T}^{\gamma}(k)=\{K_{T}^{\gamma}(k,\ell)\}_{\ell\in \Theta}$ by
\begin{equation}\label{2295kyre}
K_{T}^{\gamma}(k,\ell)=Y_{T}^{\gamma}(T+1-k,\ell),\ \ell\in\Theta,
\end{equation}
where $Y_{T}^{\gamma}(\cdot)$ is the solution of \eqref{1785ARE}.
Obviously, $K_{T}^{\gamma}(0,\ell)=0,\ \ell\in\Theta$.
And $K_{T}^{\gamma}(k)$, $k\in\mathbb{N}$ solves the following difference equation:
\begin{equation}\label{2299fordi}
\begin{split}
K^{\gamma}(k+1,\ell)=&\Psi_{1}(K^{\gamma}(k))(\ell)
-\Psi_{2}(K^{\gamma}(k))(\ell)\\
& \cdot\Psi_{3}^{\gamma}(K^{\gamma}(k))(\ell)^{-1}
\Psi_{2}(K^{\gamma}(k))(\ell)'.
\end{split}
\end{equation}
In addition, for every $T_{1}\in\mathbb{N}^{+}$ and $ T_{2}\in\mathbb{N}^{+},$
one can get that $K^{\gamma}_{T_{1}}(k)=K^{\gamma}_{T_{2}}(k),\ \forall k\in \overline{0,\min\{T_{1},T_{2}\}}$.
Therefore, we could write $K^{\gamma}(k)=\{K^{\gamma}(k,\ell)\}_{\ell\in\Theta},\ \forall k\in \mathbb{N}$
as the solution of \eqref{2299fordi} with $K^{\gamma}(0,\ell)=0,\ \ell\in\Theta$.

In Theorem \ref{2192fiboudedlemma}, the finite horizon BRL was set up based on the solution of \eqref{1785ARE}.
It is noticeable that \eqref{2295kyre} is the bridge to connect \eqref{2299fordi} with \eqref{1785ARE}.
The following lemma confirms the existence of the solution of \eqref{2299fordi}.
\begin{lemma}\label{2843gainkex}
Given $\gamma>0$, if $\Phi_{v}$ is internally stable and satisfies $\|\mathfrak{L}_{\infty}\|<\gamma$,
then \eqref{2299fordi} admits a solution $K^{\gamma}(k)\in\mathbb{H}_{\infty}^{n+}$, $k\in\mathbb{N}$
with the initial condition $K^{\gamma}(0,\ell)=0,\ \ell\in\Theta$
such that\\
(i)~$\Psi_{3}^{\gamma}(K^{\gamma}(k))(\ell)\!\leq\!-\eta_{0} I_{r\times r}$ $\mu\text{-}a.e.$, $\forall\eta_{0}\in(0,\gamma^{2}-\|\mathfrak{L}_{\infty}\|^{2})$; \\
(ii) $0\leq K^{\gamma}(k,\ell)\leq K^{\gamma}(k+1,\ell)\leq \zeta I_{n\times n}$ $\mu\text{-}a.e.$, where $\zeta>0$.
\end{lemma}
\begin{proof}
If $\|\mathfrak{L}_{\infty}\|<\gamma$,
then by Proposition \ref{2504<0lTlinf}, $\|\mathbb{L}_{T}\|<\gamma$ for every $T\in\mathbb{N}$.
Let $T\in\mathbb{N}$ be arbitrary but fixed.
Applying Theorem \ref{2192fiboudedlemma},
$Y_{T}^{\gamma}(h)\in\mathbb{H}_{\infty}^{n+}$, $\forall h\in\overline{0,T}$ is well-defined as the solution of \eqref{1785ARE}.
Then it concludes through \eqref{2295kyre} that
\eqref{2299fordi} admits a solution $K^{\gamma}(k)\in\mathbb{H}_{\infty}^{n+}$, $k\in \mathbb{N}$
with the initial condition $K^{\gamma}(0,\ell)=0,\ \ell\in\Theta$,
and $(i)$ is proved.
Next, we will prove $(ii)$.
By a similar proof with Proposition 8.5 in \cite{BookDragan2010},
we have that $K^{\gamma}(k)\leq K^{\gamma}(k+1)$, $\forall k\in\mathbb{N}$.
Consider $v_{T}=\{v_{T}(\bar{k})\}_{\bar{k}\in\overline{0,T}}\in l^{2}_{\Theta}(\overline{0,T};\mathbb{R}^{r})$
defined by $v_{T}(\bar{k})=\mathcal{F}(Y_{T}^{\gamma}(\bar{k}+1))(\vartheta(\bar{k}))x_{T}(\bar{k})$,
where $x_{T}(\bar{k})$ is the solution of
$$\left\{
\begin{array}{lr}
\!x(\bar{k}\!+\!1)\!=\![A(\vartheta(\bar{k}))\!
+\!B(\vartheta(\bar{k}))\mathcal{F}(Y_{T}^{\gamma}(\bar{k}\!+\!1))
(\vartheta(\bar{k}))]x(\bar{k}),\\
x(0)=x_{0}.\\
\end{array}
\right.$$
And set $\hat{v}_{T}=\{\hat{v}_{T}(k)\}_{k\in\mathbb{N}}\in l^{2}_{\Theta}(\mathbb{N};\mathbb{R}^{r})$ by
$\hat{v}_{T}(k)=\left\{
\begin{aligned}
v_{T}(k),& \ \  \forall k\in \overline{0, T},  \\
0,\ \  &  \ otherwise.
\end{aligned} \right.$
So for any $x_{0} \in \mathbb{R}^{n}$,
there exists $\zeta>0$ satisfying $x_{0}' Y_{T}^{\gamma}(0, \ell) x_{0}
=J_{\gamma}(T,x_{0},\ell,v_{T})\leq J_{\gamma}\left(\infty,x_{0},\ell, \hat{v}_{T}\right) \leq \zeta\left\|x_{0}\right\|^{2},$
in which the last inequality is got from Proposition \ref{2985jaccle}.
Therefore, because of \eqref{2295kyre} and the arbitrary of $T\in\mathbb{N}$,
$K^{\gamma}(k+1,\ell)\leq \zeta I_{n\times n}$ holds for every $k \in\mathbb{N}$.
\end{proof}

The infinite horizon BRL relies on the stabilizing solution $K^{\gamma}
=\{K^{\gamma}(\ell)\}_{\ell\in\Theta}
\in \mathbb{SH}_{\infty}^{n}$ satisfying the following $\Theta$-coupled ARE:
\begin{equation}\label{2970kakrq}
K^{\gamma}(\ell)\!=\!\Psi_{1}(K^{\gamma})(\ell)
\!-\!\Psi_{2}(K^{\gamma})(\ell)
\Psi_{3}^{\gamma}(K^{\gamma})(\ell)^{\!-1}
\Psi_{2}(K^{\gamma})(\ell)'.
\end{equation}
\begin{definition}\label{destasolution}
We call that $K^{\gamma}\in\mathbb{SH}_{\infty}^{n}$ is a stabilizing solution of ARE \eqref{2970kakrq}
if $\mathcal{F}(K^{\gamma})
\in \mathbb{H}_{\infty}^{r\times n}$
exponentially mean-square stabilizes $\Phi_{v}$ with conditioning,
i.e., the closed-loop system $x(k+1)=[A(\vartheta(k))+B(\vartheta(k))$ $\cdot\mathcal{F}(K^{\gamma})(\vartheta(k))]x(k)$
with the initial conditions $(x_{0},\vartheta_{0})$ is EMSS-C.
\end{definition}

We are now in a position to state the infinite horizon BRL.

\begin{theorem}\label{infbrl3404}
Given $\gamma>0$, $\Phi_{v}$ is internally stable and satisfies $\|\mathfrak{L}_{\infty}\|<\gamma$
iff \eqref{2970kakrq} admits a stabilizing solution $K^{\gamma}\in\mathbb{H}_{\infty}^{n+}$ meeting the sign condition
$\Psi_{3}^{\gamma}(K^{\gamma})(\ell)\leq-\eta_{0} I_{r\times r}\ \mu\text{-}a.e.,\ \forall\eta_{0}\in(0,\gamma^{2}-\|\mathfrak{L}_{\infty}\|^{2}).$
\end{theorem}
\begin{proof}
Sufficiency.
If \eqref{2970kakrq} admits a stabilizing solution $K^{\gamma}\in\mathbb{H}_{\infty}^{n+}$, then $\hat{K}^{\gamma}=-K^{\gamma}$ exists
as a stabilizing solution of
$$\hat{K}^{\gamma}(\ell)=\hat{\Psi}_{1}(\hat{K}^{\gamma})(\ell)
-\Psi_{2}(\hat{K}^{\gamma}) \hat{\Psi}_{3}(\hat{K}^{\gamma})(\ell)^{-1} \Psi_{2}(\hat{K}^{\gamma})(\ell)'$$
and satisfies
$\hat{\Psi}_{3}(\hat{K}^{\gamma})(\ell)\geq \eta_{0} I_{r\times r}$ $\mu\text{-}a.e.$, $\forall \eta_{0}\in(0,\gamma^{2}
-\|\mathfrak{L}_{\infty}\|^{2})$,
where $\hat{\Psi}_{1}(\cdot)(\ell):=\mathcal{T}_{A}(\cdot)(\ell)-C(\ell)'C(\ell)$
and $\hat{\Psi}_{3}(\cdot)(\ell):=\mathcal{T}_{B}(\cdot)(\ell)-D(\ell)'D(\ell)+\gamma^{2} I_{r\times r}$.
Following a similar argument as the proof of Theorem 13 in \cite{Ungureanu2013opcam},
one can yield that there are $\rho>0$, $Y \in \mathbb{SH}_{\infty}^{n}$ such that
$$\left[\begin{array}{cc}
\hat{\Psi}_{1}(Y)(\ell)-Y(\ell) & \Psi_{2}(Y)(\ell) \\
\Psi_{2}(Y)'(\ell) & \hat{\Psi}_{3}(Y)(\ell)
\end{array}\right] \geq \rho I_{(n+r)\times (n+r)},\ \ell\in\Theta.$$
And then, $Y(\ell) \leq \hat{K}^{\gamma}(\ell) \leq 0$, $\ell\in\Theta$ can be shown by a similar proof with Theorem 11 in \cite{Ungureanu2013opcam}.
Moreover, $\hat{\Psi}_{1}(Y)(\ell)-Y(\ell) \geq \rho I_{n\times n}$,
which leads to $-Y(\ell)-\mathcal{T}_{A}(-Y)(\ell) \geq \rho I_{n\times n}$ with $-Y(\ell) \geq \rho I_{n\times n}$.
Hence, by $(iv)$ of Theorem \ref{906theorem31234}, $\Phi_{v}$ is internally stable.
Next, we will prove $\left\|\mathfrak{L}_{\infty}\right\|<\gamma$.
Consider $\Phi_{v}$ with $x_{0}=0$ and let $T\in\mathbb{N}$ be arbitrary but fixed.
By completing the square associated with \eqref{2970kakrq},
\eqref{1811costf1} can be rewritten as
$J_{\gamma}(T,0,\ell,v)=-\mathbb{E}\{\phi_{x}(T\!+\!1,0,\ell,v)' K^{\gamma}(\vartheta(T\!+\!1))
\phi_{x}(T\!+\!1,0,\ell,v)|\vartheta_{0}\!=\!\ell\}$
$+\sum_{k=0}^{T} \mathbb{E}\{[v(k)\!-\hat{v}(k)]'\Psi_{3}^{\gamma}\!(K^{\gamma})(\vartheta(k))$
$\cdot[v(k)\!-\hat{v}(k)]|\vartheta_{0}\!=\ell\}\leq 0,\forall v\in l^{2}_{\Theta}(\overline{0,T};\mathbb{R}^{r}), \ \ell\in\Theta$, where $\hat{v}(k)=\!\mathcal{F}(K^{\gamma})(\vartheta(k))\phi_{x}(k,\!0,\!\ell,\!v)$.
Taking $T \rightarrow \infty$ in $J_{\gamma}\left(T,0,\ell,v\right)$ and applying Corollary \ref{co2363lims},
we have that $J_{\gamma}\left(\infty,0,\ell,v\right)=0$ iff $v(k)=\hat{v}(k)$,
which contradicts with $v\neq0$.
So, $J_{\gamma}(\infty,0,\ell,v)\!\leq\!-\eta_{0}\|v-\hat{v}\|^{2}
_{l^{2}_{\ell}(\mathbb{N};\mathbb{R}^{r})}<0\ \mu\text{-}a.e.$,
and $\left\|\mathfrak{L}_{\infty}\right\|<\gamma$ follows from Proposition \ref{pro2394LJJJ}.\\
Necessity.
In view of Lemma \ref{2843gainkex}, if $\Phi_{v}$ is internally stable and $\left\|\mathfrak{L}_{\infty}\right\|<\gamma$,
then \eqref{2299fordi} admits a solution $K^{\gamma}(k)\in\mathbb{H}_{\infty}^{n+}$
satisfying $\Psi_{3}^{\gamma}(K^{\gamma}(k))(\ell)\!\leq\!-\eta_{0} I_{r\times r}$ $\mu\text{-}a.e.$, $\forall\eta_{0}\in(0,\gamma^{2}-\|\mathfrak{L}_{\infty}\|^{2})$, and there exists $\zeta>0$ such that
$0\leq K^{\gamma}(k,\ell)\leq K^{\gamma}(k+1,\ell)\leq \zeta I_{n\times n}$ $\mu\text{-}a.e.$, $\forall k\in\mathbb{N}$.
Therefore, by the standard monotonicity result for the bounded positive semidefinite matrices,
there exists $K^{\gamma}=\{K^{\gamma}(\ell)\}_{\ell\in\Theta}\in\mathbb{H}_{\infty}^{n+}$ satisfying $\lim_{k\rightarrow\infty}{K}^{\gamma}(k,\ell)=K^{\gamma}(\ell)\leq \zeta I_{n\times n}$.
According to the dominated convergence theorem,
$\lim_{k\rightarrow\infty}\mathcal{E}({K}^{\gamma}(k))(\ell)=\int_{\Theta}g(s|\ell)
\lim_{k\rightarrow\infty}{K}^{\gamma}(k,s)\mu(ds) $ $=\mathcal{E}(K^{\gamma})(\ell).$
This concludes that $K^{\gamma}$ solves \eqref{2970kakrq} after taking the limit in \eqref{2299fordi} as $k\rightarrow\infty$.
Moreover,
$\Psi_{3}^{\gamma}(K^{\gamma})(\ell)\leq-\eta_{0} I_{r\times r}$ $\mu\text{-}a.e.$,  $\forall \eta_{0}\in(0,\gamma^{2}-\|\mathfrak{L}_{\infty}\|^{2})$.
Next, we will prove that $K^{\gamma}$ is a stabilizing solution of \eqref{2970kakrq}.
For any $v \in l^{2}_{\Theta}(\mathbb{N};\mathbb{R}^{r})$ and $\tau>0$,
introduce the perturbation operator
$\mathfrak{L}_{\infty}^{\tau}:l^{2}_{\Theta}(\mathbb{N};\mathbb{R}^{r})\rightarrow l^{2}_{\Theta}(\mathbb{N};\mathbb{R}^{m})$
as $\mathfrak{L}_{\infty}^{\tau}(v)(k,\ell)=C_{\tau}(\vartheta(k))x(k,0,\ell,v)+D_{\tau}(\vartheta(k))v(k),$
where $C_{\tau}(\cdot)=( C(\cdot)',\
                   \tau  I_{n\times n})'$ and
                   $ D_{\tau}(\cdot)=(D(\cdot)',\
                   0 )'$.
In consideration of $\left\|\mathfrak{L}_{\infty}\right\|<\gamma$,
then for a sufficiently small $\tau>0$,
$\left\|\mathfrak{L}_{\infty}^{\tau}\right\|<\gamma$.
By Lemma \ref{2843gainkex}, there exists $K^{\gamma,\tau}(k)\in\mathbb{H}_{\infty}^{n+}$, $k\in\mathbb{N}$ satisfying
\begin{equation*}
\left\{\begin{array}{l}
K^{\gamma,\tau}(k+1,\ell)=\Psi_{1}(K^{\gamma,\tau}(k))(\ell)
+ \tau^{2}  I_{n\times n}\\
\ \ -\Psi_{2}(K^{\gamma,\tau}(k))(\ell)
\Psi_{3}^{\gamma}(K^{\gamma,\tau}(k))(\ell)^{-1}
\Psi_{2}(K^{\gamma,\tau}(k))(\ell)',\\
\Psi_{3}^{\gamma}(K^{\gamma,\tau}(k))(\ell)\!
\leq-\eta_{0} I_{r\times r}, \forall\eta_{0}\in(0,\gamma^{2}\!-\!\|\mathfrak{L}^{\tau}_{\infty}\|^{2})
\end{array}\right.
\end{equation*}
with the initial condition $K^{\gamma,\tau}(0,\ell)=0$, $\ell\in\Theta$.
And for some $\zeta_{\tau}>0$,
$0\leq K^{\gamma}(k,\ell)\leq K^{\gamma,\tau}(k,\ell)\leq K^{\gamma,\tau}(k+1,\ell)\leq \zeta_{\tau} I_{n\times n}$
holds for every $k\in\mathbb{N}$, $\ell\in\Theta$,
where $K^{\gamma}(k)$ is the solution of \eqref{2299fordi}.
Therefore, $\lim_{k\rightarrow\infty}K^{\gamma,\tau}(k)
=K^{\gamma,\tau}\in\mathbb{H}_{\infty}^{n+}$ exists
and solves the following ARE:
\begin{equation}\label{2037kgta}
\begin{split}
K^{\gamma,\tau}(\ell)=&-\!\Psi_{2}(K^{\gamma,\tau})(\ell)
\Psi_{3}^{\gamma}(K^{\gamma,\tau})(\ell)^{-1}\\
\cdot\Psi_{2}&(K^{\gamma,\tau})(\ell)'
+\Psi_{1}(K^{\gamma,\tau})(\ell)\!+\tau^{2}  I_{n\times n}
\end{split}
\end{equation}
with $\Psi_{3}^{\gamma}(K^{\gamma,\tau})(\ell)\leq-\eta_{0} I_{r\times r}$
for any $\eta_{0}\in(0,\gamma^{2}-\|\mathfrak{L}^{\tau}_{\infty}\|^{2}).$
Moreover, $K^{\gamma}\leq K^{\gamma,\tau}$.
Subtracting \eqref{2970kakrq} from \eqref{2037kgta}, it follows that
$K^{\gamma,\tau}(\ell)-K^{\gamma}(\ell)
=[A(\ell)+B(\ell)\mathcal{F}(K^{\gamma})(\ell)]'\mathcal{E}(K^{\gamma,\tau}\!-\!K^{\gamma})(\ell)
[A(\ell)+B(\ell)\mathcal{F}(K^{\gamma})(\ell)]+M(\ell),$
where $M(\ell)\!=\!\tau^{2} I_{n\times n}-[\mathcal{F}(K^{\gamma})(\ell)-\mathcal{F}(K^{\gamma,\tau})(\ell)]'$
$\cdot\Psi_{3}^{\gamma}(K^{\gamma,\tau})(\ell)[\mathcal{F}(K^{\gamma})(\ell)-\mathcal{F}(K^{\gamma,\tau})(\ell)]$ $\geq\tau^{2} I_{n\times n}$.
This leads to $K^{\gamma,\tau}(\ell)-K^{\gamma}(\ell)\geq \tau^{2} I_{n\times n}$.
By Theorem \ref{906theorem31234},
$F(K^{\gamma})$ exponentially stabilizes $\Phi_{v}$,
i.e., $K^{\gamma}$ is a stabilizing solution of \eqref{2970kakrq}.
\end{proof}
\begin{remark}
For finite MJLSs, with the intention of deriving the infinite horizon BRL,
it is typical in previous literature to assume the weak controllability of the system (see, \cite{Seiler2003,Aberkane2013})
or that the transition probability matrix of the Markov chain is nondegenerate (see, \cite{Collin2016,Marcos2018auto}).
Note that the latter is less restrictive than the former.
In this paper, we adopt Assumption \ref{1977assgtl0},
which corresponds to the transition probability matrix being nondegenerate in the finite scenario.
\end{remark}

In what follows, a backward iteration will be presented for solving \eqref{2970kakrq},
whenever \eqref{2970kakrq} admits a stabilizing solution satisfying the sign condition in Theorem \ref{infbrl3404}.
To this end, for every $T\in\mathbb{N}$, we consider the backward difference equation \eqref{1785ARE}.
Applying Theorem \ref{infbrl3404}, we confirm that $(\mathbf{A};\mathcal{G})$ is EMSS-C and $\Phi_{v}$ satisfies $\|\mathfrak{L}_{\infty}\|<\gamma$.
Following a similar argument as the proof of Lemma \ref{2843gainkex},
it can be concluded that \eqref{1785ARE} admits a solution $Y_{T}^{\gamma}(k),\ k\in{\overline{0, T}}$ satisfying the sign condition $\Psi_{3}^{\gamma}(Y_{T}^{\gamma}(k))(\ell)\leq-\eta_{0} I_{r\times r}$, $\ell\in\Theta$.
Moreover, $0\leq Y_{T}^{\gamma}(k,\ell)\leq Y_{T+1}^{\gamma}(k,\ell)\leq \zeta I_{n\times n}$ $\mu\text{-}a.e.$,
where $Y_{T+1}^{\gamma}(k),\ k\in \overline{0,T+1}$ is the solution of \eqref{1785ARE} with $Y(T+2,\ell)=0$ and $\zeta>0$.
Therefore, $\{Y_{T}^{\gamma}(0)\}_{T\in \mathbb{N}}$ is a monotone increasing sequence which is bounded above.
It then follows from the standard monotonicity result concerning the bounded positive semidefinite matrices
that $\lim_{T\rightarrow+\infty}Y_{T}^{\gamma}(0)=\lim_{T\rightarrow+\infty}Y_{T+1}^{\gamma}(0)=Y^{\gamma}$.
Take into account the difference equation  $Y_{T+1}^{\gamma}(0)=\Psi_{1}(Y_{T+1}^{\gamma}(1))
-\Psi_{2}(Y_{T+1}^{\gamma}(1))\Psi_{3}^{\gamma}(Y_{T+1}^{\gamma}(1))^{-1}\Psi_{2}(Y_{T+1}^{\gamma}(1))'$.
Since $Y^{\gamma}_{T+1}(1)=Y^{\gamma}_{T}(0),\ \forall T\in\mathbb{N}$,
taking $T\rightarrow\infty$,
we know that $Y^{\gamma}$ satisfies \eqref{2970kakrq}.
From what has been discussed above,
we propose a backward iteration summarized as Algorithm \ref{alg1}
for getting a numerical solution of a positive semidefinite solution to \eqref{2970kakrq}.
\begin{algorithm}
\caption{An algorithm for solving \eqref{2970kakrq}}\label{alg1}
\hspace*{0.02in} {\bf Input:}
the computational accuracy $\epsilon>0$ and a prescribed level $\gamma>0$ \\
\hspace*{0.02in} {\bf Output:}
a numerical solution of \eqref{2970kakrq}
\begin{algorithmic}[1]
\State For any $\ell\in\Theta$ set $\mathfrak{T}=1$ and $Z(\ell)=0$.
\State Set $\mathfrak{T}_{Y}=\mathfrak{T}$, $\mathfrak{T}_{X}=\mathfrak{T}+1$, and $Y(\ell)=Z(\ell)$.
\State Compute $Y(\ell)=\Psi_{1}(Y)(\ell)-\Psi_{2}(Y)(\ell)\Psi_{3}^{\gamma}(Y)(\ell)^{-1}\Psi_{2}(Y)(\ell)'.$
\State $\mathfrak{T}_{Y}=\mathfrak{T}_{Y}-1.$
\State If $\mathfrak{T}_{Y}>0$ then go to Step 3. Else continue to Step 6.
\State Set $X(\ell)=Z(\ell)$.
\State Compute $X(\ell)=\Psi_{1}(X)(\ell)-\Psi_{2}(X)(\ell)\Psi_{3}^{\gamma}(X)(\ell)^{-1}\Psi_{2}(X)(\ell)'.$
\State $\mathfrak{T}_{X}=\mathfrak{T}_{X}-1.$
\State If $\mathfrak{T}_{X}>0$ then go to Step 7. Else continue to Step 10.
\State If $\|Y(\ell)-X(\ell)\|<\epsilon$ for any $\ell\in\Theta$ then return $X(\ell)$ and halt.
Else set $\mathfrak{T}=\mathfrak{T}+1$ and go to step 2.
\end{algorithmic}
\end{algorithm}
\begin{remark}
It should be noted that Algorithm \ref{alg1} is efficient to implement,
but it has a limitation since it only solves the numerical solution of a positive semidefinite solution
to \eqref{2970kakrq} rather than the stabilizing solution.
Regarding to finite MJLSs,
one can refer to \cite{2022Game} for an iterative deterministic algorithm of the stabilizing solutions to a class of stochastic Riccati equations,
in which the original problem is transformed into solving a sequence of uncoupled deterministic Riccati equations.
However, for the ARE \eqref{2970kakrq} coupled via the integral, it is significantly challenging to obtain the auxiliary uncoupled Riccati equations. Additionally, we recognize that the existence conditions for the stabilizing solution of the ARE is an issue of considerable importance.
\cite{2020AberkaneJDEA} has made a great progress for finite MJLSs.
In future research, we hope to do some in-depth studies on the conditions for the existence of the stabilizing solution of the ARE
and the numerical algorithm for the case where the Markov chain takes values in $\Theta$.
We thank the anonymous reviewer for providing valuable research ideas for us.
\end{remark}

\section{Examples}\label{Example}

In this section, several examples are supplied to illustrate the feasibility of our work.
\begin{example}\label{ex1}
The results of analysis and synthesis of discrete-time MJLSs,
where the Markov chain takes values in $\Theta$,
potentially provide significant value for some models constructed based on actual applications (see \cite{Meynsp2009}).
For example, it has been applied to networked control systems with the continuous-valued random delays (see, \cite{Masashi2018}).
In this example, we consider the solar thermal receiver model proposed in \cite{Sworder1983},
which was also discussed in Chapter 8 of \cite{BookCosta2005}.
In these works, the influence of atmospheric conditions on the system dynamics is usually modeled by a two-mode Markov chain,
representing two atmospheric conditions: 1) sunny and 2) cloudy.
However, even under the sunny or cloudy atmospheric condition,
the instantaneous solar radiation still fluctuates and then heavily influences the system dynamics.
Therefore, it is reasonable to establish a more accurate model by using a Markov chain that takes values in $\Theta$
to describe the effect of varying solar radiation intensities on the system dynamics.
In \cite{Costa2015}, the system dynamics is modeled by $x(k+1)=a(\vartheta(k))x(k),\ k\in\mathbb{N}$ and the Markov chain $\{\vartheta(k)\}_{k\in\mathbb{N}}$ is on the Borel space $(\Theta, \mathcal{B}(\Theta))$,
where $\Theta=\Theta_{1}\bigcup\Theta_{2}$
with $\Theta_{i}=\{i\times [0,1]\}$, $i\in\{1,2\}$.
For $\vartheta(k)=(i,t)$ with $t\in[0,1]$, $i=1$ $(i=2)$ represents the sunny (cloudy) atmospheric condition,
while $t$ represents the effects of instantaneous insolation on the system parameters under the sunny (cloudy) atmospheric condition.
Define the measure $\mu$ on $\mathcal{B}(\Theta)$ as $\mu(\{i\}\times M)=\mu_{L}(M)$,
where $\mu_{L}(\cdot)$ is the Lebesgue measure, $M$ is a Borel set on $[0,1]$ and $i\in\{1,2\}$.
When $\vartheta(k)=(i,t)$ and $\vartheta(k+1)=(j,S)$ with $S$ being uniformly distributed over the interval $[0,1]$,
let $j=1$ with probability $P_{i1}$ and $j=2$ with  probability $P_{i2}$ be the transition probability function.
Consider the parameters given in \cite{Costa2015}: $a_{11}=0.9,\ a_{12}=0.7,\ a_{21}=0.95,\ a_{22}=1.15$.
And for $a=\{a(\ell)\}_{\ell\in\Theta}$, when $\ell=(i,t)$,
set $a(\ell)=a_{i1}+t(a_{i2}-a_{i1})$, $i\in\{1,2\}$, $t\in[0,1]$.
Here, take $P_{11}=0.9767,\ P_{22}=0.7611$.
Since $U(\ell)=[a(\ell)]^{2}$ and $\xi=0.0047$ solves $U(\ell)-\mathcal{T}_{a}(U)(\ell)\geq \xi $ $\mu\text{-}a.e.$,
from (iv) of Theorem \ref{906theorem31234}, one can infer that the system is EMSS-C.
\end{example}
\begin{example}\label{2558ex1}
First, we consider a two-dimension MJLS $(\mathbf{A},P)$: $x(k+1)=A(\vartheta(k))x(k),\ k\in\mathbb{N}$
and the Markov chain $\{\vartheta(k)\}_{k\in\mathbb{N}}$ takes values in a finite set $\{1,2\}$.
Define the measure on $\{1,2\}$ as the counting measure,
and give the probability transition matrix of the Markov chain by $P=(p_{ij})_{2\times 2}$ with $p_{11}=0.15$ and $p_{22}=0.1$.
The system coefficients are taken as $A(\vartheta(k))=A_{11}$ for $\vartheta(k)=1$ and $A(\vartheta(k))=A_{21}$ for $\vartheta(k)=2$,
where
$A_{11}=
\left(
   \begin{array}{cc}
      2 & -1 \\
          0 & 0 \\
      \end{array}
\right),
A_{21}=
\left(
   \begin{array}{cc}
      0 & 1 \\
     0 & 2\\
      \end{array}
\right).$
It is not difficult to know that there are instable individual modes for the finite MJLS $(\mathbf{A},P)$.
However, we can find that $Y=\{Y(i)\}_{i\in\overline{1,2}}$ with $Y(i)=Y_{i}\in\mathbb{S}^{2+*}$ solves the Lyapunov-type inequality:
$-Y(i)+A(i)'[\sum_{j=1}^{2} p_{ij}Y(j)]A(i)+I_{2\times 2}\leq 0, \ \forall i\in\overline{1,2},$
where
$Y_{1}=10^{3}\times\left(
                \begin{array}{cc}
                    1.3438  & -0.6177\\
                   -0.6177  &  0.4501\\
                \end{array}
              \right),
Y_{2}=10^{3}\times\left(
                \begin{array}{cc}
                 0.1104  & -0.0044\\
                 -0.0044 &   1.3873\\
                \end{array}
              \right).$
According to (iv) of Theorem \ref{906theorem31234}, it follows that $(\mathbf{A},P)$ is EMSS-C,
thereby $(\mathbf{A},P)$ is EMSS by Proposition \ref{co973EMSSEMSSC}.
Or directly from Theorem \ref{spectralcriterion},
the EMSSy of $(\mathbf{A},P)$ is immediately available since $r_\sigma(\mathcal{L}_{A})=0.6<1$.
Therefore, the instability of individual modes does not imply that the finite MJLS is not EMSS
(one can refer to \cite{Fangyg2002} for other examples).

Next, we consider a MJLS with the Markov chain on a Borel space,
which can be viewed as the uncountably infinite scenario corresponding to $(\mathbf{A};P)$ discussed above.
Consider the  MJLS $(\mathbf{A};\mathcal{G})$ with $\{\vartheta(k)\}_{ k\in\mathbb{N}}$ being on $(\Theta, \mathcal{B}(\Theta))$,
where $\Theta$ and the measure $\mu$ are the same as those in Example \ref{ex1}.
The transition probability function is given as follows:
$\mathcal{P}\{\vartheta(k+1)\in\Theta_{j}|\vartheta(k)\in\Theta_{i}\}=p_{ij}$, $i\in\overline{1,2}$, $j\in\overline{1,2}$.
As for $A=\{A(\ell)\}_{\ell\in\Theta}$ of $(\mathbf{A};\mathcal{G})$,
let $A(\ell)=A(1,t)=A_{11}+t(A_{12}-A_{11})$ for $\ell=(1,t)\in\Theta_{1}$
and $A(\ell)=A(2,t)=A_{21}+t(A_{22}-A_{21})$ for $\ell=(2,t)\in\Theta_{2}$, $0\leq t\leq 1$,
where $A_{12}=\frac{1}{2}A_{11}$, $A_{22}=\frac{1}{2}A_{21}$.
Although there are instable individual modes,
we find that $Y=\{Y(\ell)\}_{\ell\in\Theta}
\in\mathbb{H}_{\infty}^{2+*}$ with $Y(\ell)=Y(i,t)$ for $\ell=(i,t)$
solves $-Y(i,t)+A(i,t)'[\sum_{j=1}^{2}\int^{1}_{0} p_{ij}Y(j,s)ds]A(i,t)+I_{2\times 2}\leq 0,$ $\forall i\in\overline{1,2},\ t\in[0,1],$
where $Y(1,t)=Y_{1}$ and $Y(2,t)=Y_{2}$.
From (iv) of Theorem \ref{906theorem31234}, $(\mathbf{A};\mathcal{G})$ is EMSS-C.
Furthermore, by Proposition \ref{co973EMSSEMSSC}, we deduce that $(\mathbf{A};\mathcal{G})$ is EMSS.
Hence, for the case where the Markov chain takes values in $\Theta$,
the instability of individual modes does not imply that the MJLS is not EMSS,
which is consistent with the finite case.
\end{example}
\begin{example}
In this example, consider a two-dimension MJLS $\Phi_{v}$,
where $\{\vartheta(k)\}_{k\in\mathbb{N}}$ and $A$ are the same as those of $(\mathbf{A};\mathcal{G})$ in Example \ref{2558ex1}.
It has been shown in Example \ref{2558ex1} that $(\mathbf{A};\mathcal{G})$ is EMSS-C.
Set $B(\ell)=tB_{i}$, $C(\ell)=tC_{i}$, $D(\ell)=0$ for $\ell=(i,t)\in\Theta_{i}$, $i\in\overline{1,2}$, $t\in[0,1]$, where
$B_{1}\!=\!(\begin{array}{cc}
         0.4 & -0.2
       \end{array})',$
$C_{1}\!=\!(\begin{array}{cc}
         -0.12 & -0.3
       \end{array}),$
$B_{2}\!=\!(\begin{array}{cc}
         -1.1 & 0.3
       \end{array})'$,
$C_{2}=(\begin{array}{cc}
         0.4 & 0.2
       \end{array}).$
And take $\gamma=0.5$.
Applying Algorithm \ref{alg1} with the computational accuracy $\epsilon=10^{-5}$
and considering the finite grid approximation for $\Theta$,
we get a numerical solution $K=\{K(\ell)\}_{\ell\in\Theta}$ of \eqref{2970kakrq} presented in Fig.\ref{Figure1},
in which $K(\ell)=K(i,t)$ is a positive semidefinite matrix for any $\ell=(i,t)\in\Theta_{i}$, $i\in\{1,2\}$, $t\in[0,1]$.
Moreover, $K$ satisfies $B(i,t)'[\sum_{j=1}^{2}\int^{1}_{0} p_{ij}K(j,s)ds]B(i,t)-\gamma^{2}<0$.
Substituting $K$ into $F=\mathcal{F}(K)$,
it can also be verified that $Y=\{Y(\ell)\}_{\ell\in\Theta}$
with $Y(\ell)=Y_{i}\in\mathbb{S}^{2+*}$ for $\ell=(i,t)\in\Theta_{i}$ as given in Example \ref{2558ex1},
satisfies
$-Y(i,t)+[A(i,t)+B(i,t)F(i,t)]'[\sum_{j=1}^{2}\int^{1}_{0} p_{ij}Y(j,s)ds]$ $\cdot[A(i,t)+B(i,t)F(i,t)]+I_{2\times 2}\leq 0$,
$i\in\{1,2\}$, $t\in[0,1]$.
These ensure that $K$ is the stabilizing solution of \eqref{2970kakrq} and satisfies the required sign condition.

Now, let us take an exogenous disturbance signal $v=\{v(k)\}_{k\in\mathbb{N}}$ with finite energy, where $v(k)=e^{-2k}$.
Fig. \ref{Figure2} demonstrates 100 possible trajectories of $\Phi_{v}$,
and the mean value of the system state is marked.
During this simulation, the sampling period $d$ is 0.01 hours.
From Fig. \ref{Figure3}, one can see that for any sample time $s\in\{d,\ 2d,\ 3d,\cdots, 400d\}$,
$\|y\|_{l^{2}(0,s;R)}/\|v\|_{l^{2}(0,s;R)}\leq\gamma$,
which verifies the feasibility of Theorem \ref{infbrl3404}.
\begin{figure}[!htb]
\centering
\includegraphics[width=0.5\textwidth]{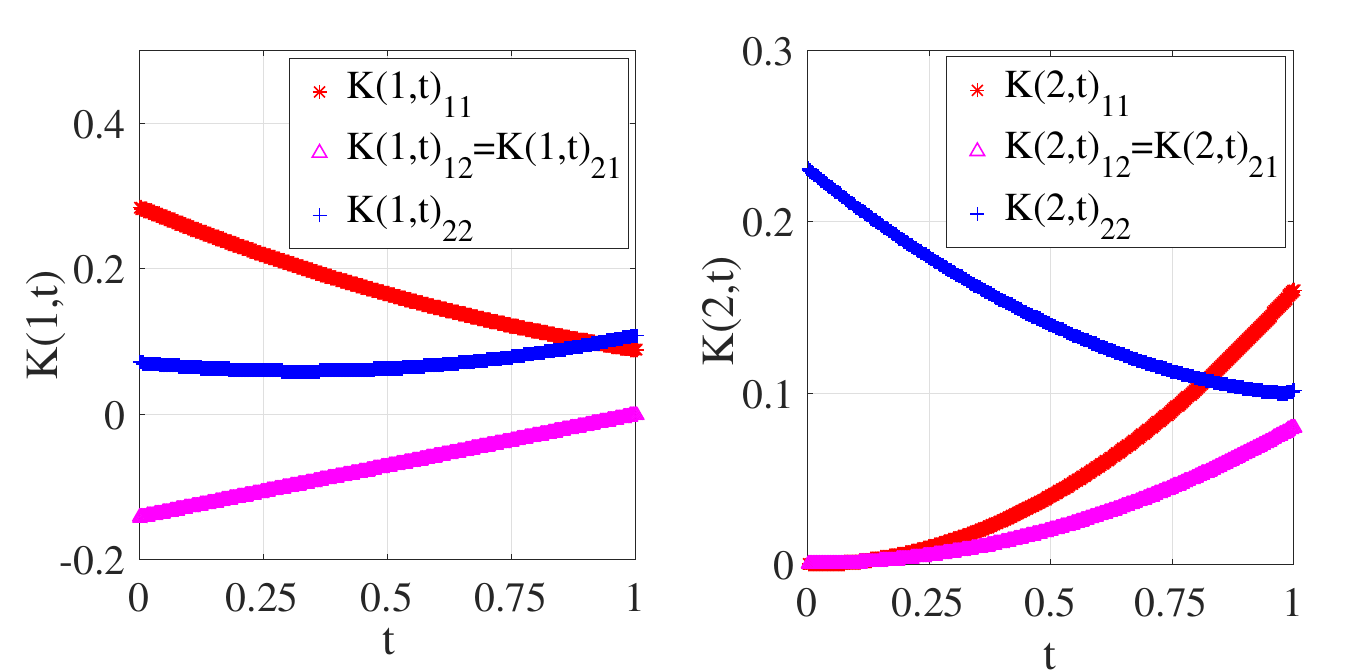}\\
\caption{\small The numerical solution of \eqref{2970kakrq}. $K(1,t)$ $(K(2,t))$ is plotted on the left (right),
where $K(i,t)_{hl}$ is the $(h,l)$-th entry of matrix $K(i,t)$, $h\in\{1,2\},\ l\in\{1,2\}$.}\label{Figure1}
\end{figure}
\begin{figure}[!htb]
\centering
\includegraphics[width=0.5\textwidth]{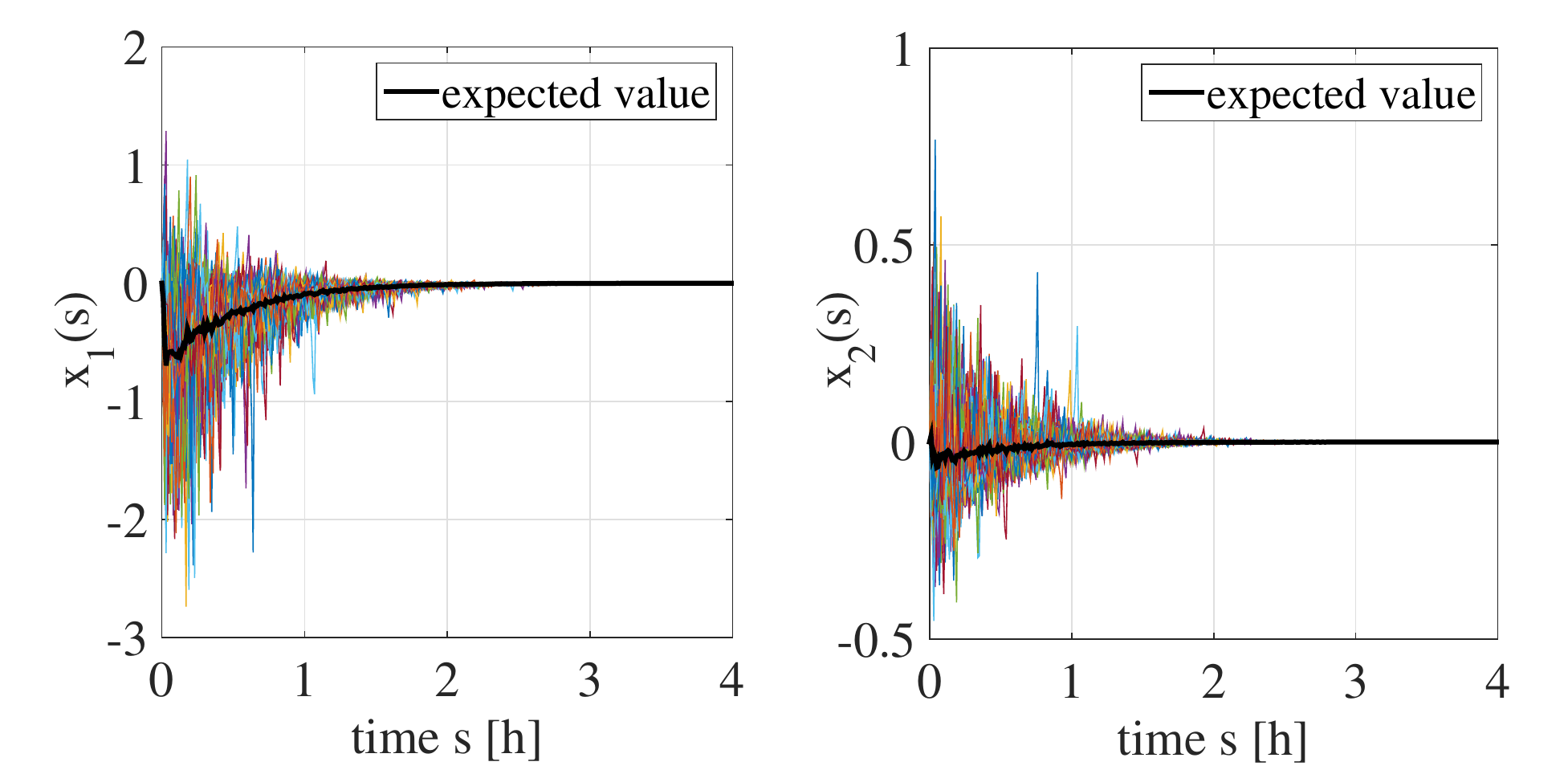}\\
\caption{\small
100 possible trajectories of $x_{1}(s)$ $(x_{2}(s))$ are plotted on the left (right).}\label{Figure2}
\end{figure}
\begin{figure}[!htb]
\centering
     \includegraphics[width=0.5\textwidth]{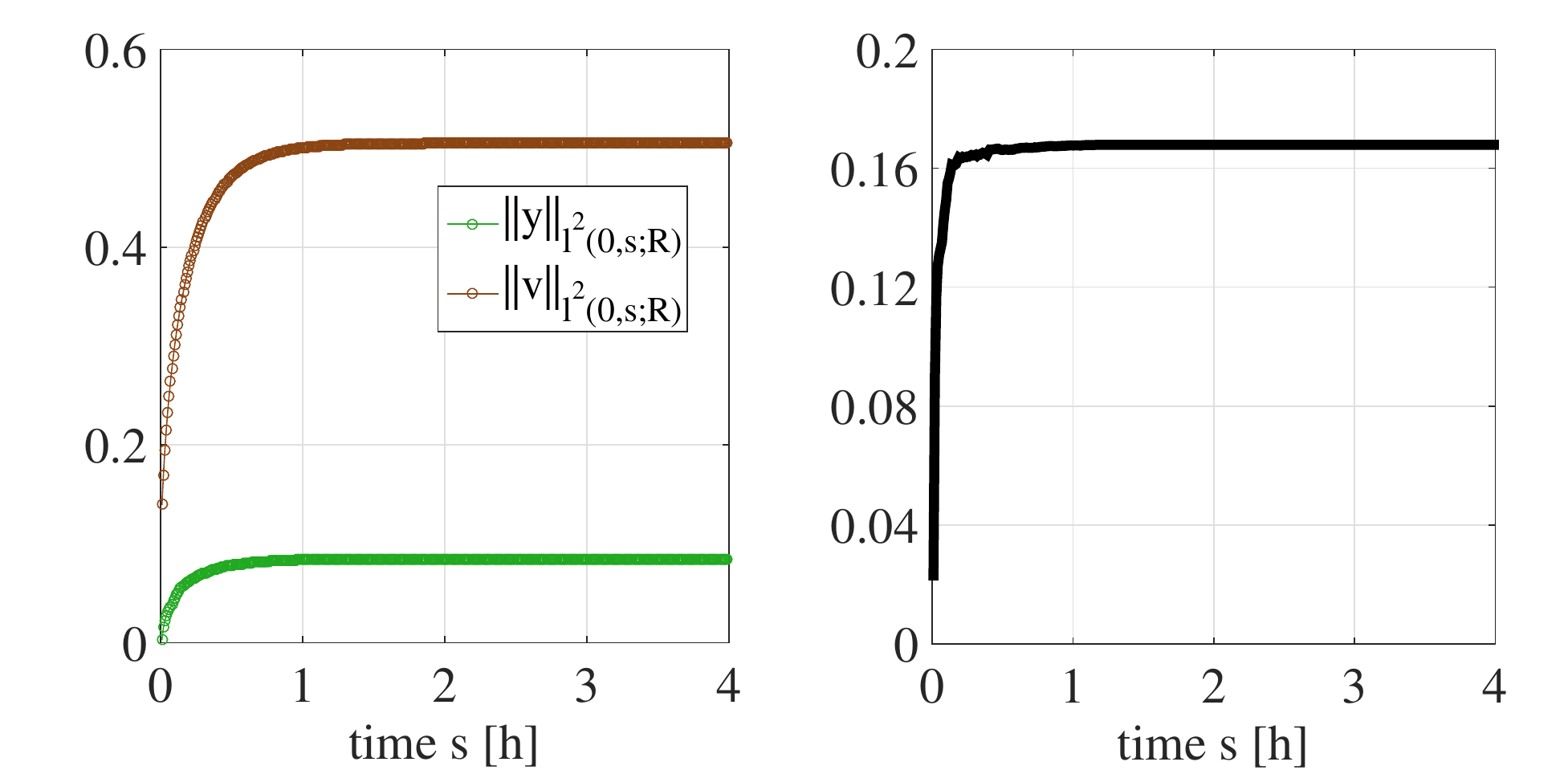}\\
\caption{\small Let $\|v\|_{l^{2}(0,s;R)}$ $(\|y\|_{l^{2}(0,s;R)})$ be the cumulative energy of
the disturbance signal $v$ (output $y$) for $\Phi_{v}$.
The trajectories of $\|v\|_{l^{2}(0,s;R)}$ and $\|y\|_{l^{2}(0,s;R)}$ are plotted on the left,
and the ratio $\|y\|_{l^{2}(0,s;R)}/ \|v\|_{l^{2}(0,s;R)}$ is plotted on the right. }\label{Figure3}
\end{figure}
\end{example}

\section{Conclusion}\label{Conclusion}

This paper has studied exponential stability and the disturbance attenuation property
for discrete-time MJLSs with the Markov chain on a Borel space $(\Theta, \mathcal{B}(\Theta))$.
The results developed could be viewed as an extension of the previous ones
for MJLSs with the Markov chain on a countable set to the uncountable scenario.
First, two kinds of exponential stabilities have been introduced: EMSSy and EMSSy-C.
The spectral criteria have been proposed for EMSSy and EMSSy-C respectively,
and the Lyapunov-type criteria have also been established for EMSSy-C.
As for the relationship between these two kinds of exponential stabilities,
we are of the opinion that EMSSy-C and EMSSy are equivalent (see Conjecture \ref{conjecture}),
and some discussions have been made in Proposition \ref{co973EMSSEMSSC} and Remark \ref{797reborenoeq}.
Regarding the disturbance attenuation property,
the BRLs have been presented in terms of the $\Theta$-coupled DRE for the finite horizon case
and ARE for the infinite horizon case.
As the core of $H_{\infty}$ analysis, the BRL will play a significant role in more studies on the system synthesis issues,
such as $H_{\infty}$ filter and $H_{\infty}$ controller designs for MJLSs.
In addition, it is worthwhile to note that the results drawn in this paper on exponential stability and BRLs
can be directly extended to the system with multiple noises.

\end{document}